\documentclass{amsart}
\usepackage{amsmath}
\usepackage{amsthm,amscd,amssymb,amsfonts, amsbsy}
\usepackage{latexsym}
\usepackage{txfonts}
\usepackage{mathrsfs}

\usepackage{color}
\usepackage{enumerate}

\numberwithin{equation}{section}

\theoremstyle{plain}
\newtheorem{theorem}[equation]{Theorem}
\newtheorem{lemma}[equation]{Lemma}

\theoremstyle{definition}

\newtheorem*{acknowledgment}{Acknowledgment}
\newtheorem*{CLB}{Condition (LB)}
\newtheorem*{CLBp}{Condition (LB$'$)}
\newtheorem*{CIH}{Condition (IH)}
\newtheorem*{CIHp}{Condition (IH$'$)}

\theoremstyle{remark}

\newtheorem{remark}[equation]{Remark}

\newcommand{\supp}{\operatorname{supp}}
\newcommand{\dist}{\operatorname{dist}}
\newcommand{\diam}{\operatorname{diam}}

\newcommand{\mysection}[1]{\section{#1}
\setcounter{equation}{0}}

\renewcommand{\vec}[1]{\boldsymbol{#1}}

\newcommand{\bR}{\mathbb R}

\newcommand\Lt{{}^t\!L}

\newcommand{\LB}{\mathrm{(LB)}}
\newcommand{\IH}{\mathrm{(IH)}}

\newcommand{\ip}[1]{\left\langle#1\right\rangle}

\providecommand{\set}[1]{\{#1\}}

\providecommand{\abs}[1]{\lvert#1\rvert}
\providecommand{\Abs}[1]{\left\lvert#1\right\rvert}
\providecommand{\bigabs}[1]{\bigl\lvert#1\bigr\rvert}

\providecommand{\Biggabs}[1]{\Biggl\lvert#1\Biggr\rvert}

\providecommand{\norm}[1]{\lVert#1\rVert}

\renewcommand{\epsilon}{\varepsilon}
\renewcommand{\qedsymbol}{$\blacksquare$}

\begin{document}
\title[Neumann functions]{Neumann functions for second order elliptic systems with measurable coefficients}

\author[J. Choi]{Jongkeun Choi}
\address[J. Choi]{Department of Mathematics, Yonsei University, Seoul 120-749, Republic of Korea}
\email{cjg@yonsei.ac.kr}

\author[S. Kim]{Seick Kim}
\address[S. Kim]{Department of Mathematics, Yonsei University, Seoul 120-749, Republic of Korea}
\curraddr{Department of Computational Science and Engineering, Yonsei University, Seoul 120-749, Republic of Korea}
\email{kimseick@yonsei.ac.kr}


\begin{abstract}
We study Neumann functions for divergence form, second order elliptic systems with bounded measurable coefficients in a bounded Lipschitz domain or a Lipschitz graph domain.
We establish existence, uniqueness, and various estimates for the Neumann functions under the assumption that weak solutions of the system enjoy interior H\"older continuity.
Also, we establish global pointwise bounds for the Neumann functions under the assumption that weak solutions of the system satisfy a certain natural local boundedness estimate.
Moreover, we prove that such a local boundedness estimate for weak solutions of the system is in fact equivalent to the global pointwise bound for the Neumann function.
We present a unified approach valid for both the scalar and the vectorial cases.
\end{abstract}

\maketitle

\mysection{Introduction}	
 
In this article, we are concerned with Neumann functions (or sometimes called Neumann Green's function) for divergence form, second order elliptic systems with bounded measurable coefficients in a bounded Lipschitz domains or a Lipschitz graph domain.
More precisely, we consider Neumann functions for the $m\times m$ elliptic systems
\begin{equation}    \label{eq0.0}
\sum_{j=1}^m L_{ij} u^j:=-\sum_{j=1}^m\sum_{\alpha,\beta=1}^d
D_\alpha(A^{\alpha\beta}_{ij}(x) D_\beta u^j), \quad i=1,\ldots,m
\end{equation}
in $\Omega$, where $\Omega$ is a bounded Lipschitz domain or a Lipschitz graph domain in $\bR^d$ with $d\ge 3$.
Here, we assume that the coefficients are measurable functions defined in the whole space $\bR^d$ satisfying the strong ellipticity and the uniform boundedness condition; see Section~\ref{pre} for their precise definitions.
We do not assume that the coefficients of the system \eqref{eq0.0} are symmetric.
We will later impose some further assumptions on the system \eqref{eq0.0} in the case when $m>1$ but not explicitly on its coefficients.

Analogous to the role of Green's functions in the study of Dirichlet boundary value problem of elliptic equations, Neumann functions play a significant role in the study of Neumann boundary value problem.
By this reason, the Neumann functions are discussed in many papers, but however, with only a few exceptions, it is assumed that the coefficients and the domains are sufficiently regular.
In the case when $m=1$, Kenig and Pipher \cite{KP93} constructed Neumann functions for the divergence form elliptic equations with $L^\infty$ coefficients and derived various estimates for the Neumann functions in the unit ball $B$.
Those estimates are the same sorts of estimates known for the Green's functions as appear in \cite{GW, LSW} and are nicely summarized in \cite[Theorem~1.6.3]{Kenig94}.
Their methods of proof are general enough to allow $B$ to be a bounded star-like Lipschitz domain but however, as pointed out in \cite{Shen04}, it is not immediately clear whether they also work for general bounded Lipschitz domains.
Also, their methods do not seem to work for unbounded domains such as the half space.
On the other hand, Hofmann and Kim \cite{HK07} recently proved existence and various (interior) estimates for the Green's function of the system \eqref{eq0.0} in arbitrary domains under the assumption that weak solutions of the system \eqref{eq0.0} satisfy an interior H\"older continuity estimate.
Their result has been complemented by a very recent article by Kang and Kim \cite{KK10}, where global estimates of Green's functions for the system \eqref{eq0.0} are established under some other (but similar) assumptions.
In the case when $m=1$, the De Giorgi-Moser-Nash theory for weak solutions implies such estimates and thus, in particular, they were able to reproduce the related classical results of \cite{GW, LSW}.

The goal of this article is to present a unified approach for the construction and estimates of Neumann functions of the elliptic systems \eqref{eq0.0} in a bounded Lipschitz domain as well as in an unbounded domain above a Lipschitz graph.
As a matter of fact, it is exactly where the strength of our paper lies.
By using our unified method, we reproduce the estimates for Neumann functions of scalar equations with $L^\infty$ coefficients in the unit ball presented in \cite{KP93} as well as those for systems with $C^\alpha$ coefficients in $C^{1,\alpha}$ domains appearing in a recent article \cite{KLS}.
Recently, there have been some interest in studying boundary value problems for divergence form elliptic equations with complex $L^\infty$ coefficients above a Lipschitz graph; see e.g., \cite{AAAHK, AA11, AAH}.
In this context, it is natural to consider Green's functions and Neumann functions for elliptic systems with $L^\infty$ coefficients in a Lipschitz graph domain.
In fact, properties of Green's function investigated in \cite{HK07}, for elliptic equations whose coefficients are complex perturbations of real $L^\infty$ coefficients, were used in \cite{AAAHK}.
However, we are not even able to find a literature dealing with Neumann functions in the half space for scalar elliptic equations with $L^\infty$ coefficients.
As we have already mentioned, our method also goes through in that case, and in particular, we derive the estimates for the Neumann function of the scalar elliptic equations with $L^\infty$ coefficients in a Lipschitz graph domain that corresponds to the estimates in \cite[Theorem~1.6.3]{Kenig94}.
We hope this article may serve as a reference for the Neumann functions and their properties, so that it may become a useful tool for other authors.

We shall now describe our main result briefly.
Let $\Omega\subset \bR^d$ with $d\ge 3$ be a bounded Lipschitz domain or a Lipschitz graph domain.
We first construct the Neumann function of the system \eqref{eq0.0} under the assumption that its weak solutions are locally H\"older continuous.
In doing so, we also derive various interior estimates for the Neumann function; see Theorem~\ref{thm1} and \ref{thm4}.
We then show that if the system \eqref{eq0.0} has such a property that  weak solutions of Neumann problems with nice data are locally bounded and satisfy a certain natural estimate (see the conditions \eqref{LB} and \eqref{LBp} in Section~\ref{main} and \ref{lgd}), then its Neumann function $\vec N(x,y)$ has the following global pointwise bound; see Theorem~\ref{thm2} and \ref{thm5}.
\begin{equation}					\label{eq0.1a}
\abs{\vec N(x,y)} \le C \abs{x-y}^{2-d},\quad \forall x,y \in\Omega,\quad x\ne y.
\end{equation}
Conversely, if the Neumann function has the above pointwise bound, then we prove that the system should satisfy the aforementioned local boundedness property; see Theorem~\ref{thm3} and \ref{thm5}.
An immediate consequence of our results combined with the celebrated De Giorgi-Moser-Nash theory would be that the Neumann function of scalar elliptic equations (i.e., $m=1$) enjoy the pointwise estimate \eqref{eq0.1a} if $\Omega$ is a bounded Lipschitz domain or a Lipschitz graph domain.
Moreover, if the coefficients of the system \eqref{eq0.0} belong to the VMO class and $\Omega$ is a bounded $C^1$ domain, then  $W^{1,p}$ estimates imply the aforementioned local boundedness property and thus, we would have the estimate \eqref{eq0.1a} in that case too.
As a matter of fact, in those cases, we also have
\begin{align*}
\abs{\vec N(x,y)-\vec N(x',y)} &\le C \abs{x-x'}^\mu \abs{x-y}^{2-d-\mu}\;\text{ if }\;x\ne y\;\text{ and }\;2\abs{x-x'}<\abs{x-y},\\
\abs{\vec N(x,y)-\vec N(x,y')} &\le C \abs{y-y'}^\mu \abs{x-y}^{2-d-\mu}\;\text{ if }\;x\ne y\;\text{ and }\;2\abs{y-y'}<\abs{x-y}
\end{align*}
for some $\mu\in (0,1]$; see Remark~\ref{rmk3.3}.

The organization of the paper is as follows.
In Section~\ref{pre}, we introduce some notation and definitions including weak formulations of Neumann problems and the precise definition of Neumann functions of the system \eqref{eq0.0}.
In Section~\ref{main}, we state our main theorems including existence and global pointwise estimates for Neumann functions in bounded Lipschitz domains, and their proofs are presented in Section~\ref{pf}.
Section~\ref{lgd} is devoted to the study of Neumann function in a Lipschitz graph domain.
In the appendix we provide the proofs of some technical lemmas.

Finally, a few remarks are in order.
This article is, in spirit, very similar to \cite{HK07, KK10}, where corresponding results for Green's functions have been established.
However, the technical details are very different since Neumann boundary condition is more difficult to handle than the Dirichlet condition.
For instance, in \cite{HK07}, the Green's functions are constructed in arbitrary domains but here Neumann functions are constructed only in domains with Lipschitz boundary. 
We do not treat the case $d=2$ in our paper.
In dimension two, the Neumann functions should have logarithmic growth and requires some other methods.
As a matter of fact, our method breaks down and is not applicable in two dimensional case.
One way to overcome this difficulty is to utilize so-called Neumann heat kernel of the elliptic operator defined in a Lipschitz cylinder $\Omega\times (0,\infty)\subset \bR^3$.
However, this approach requires first establishing a pointwise bound for Neumann heat kernel that is sharp enough to be integrable in $t$-variable;  see \cite{CDK11, DK09} for the treatment of Green's functions of elliptic systems in two dimensional domains.
This topic will be discussed elsewhere because Neumann heat kernel is an interesting subject in its own right.
After submission of the first version of this paper, Taylor et al. \cite{TKB} constructed the Green's function for the mixed problem for elliptic systems in two dimensions.

\mysection{Preliminaries}					\label{pre}
\subsection{Basic Notation}
We mainly follow the notation used in \cite{HK07, KK10}.
Let $d\ge 3$ be an integer.
We recall that a function $\varphi:\bR^{d-1}\to \bR$ is Lipschitz if there exists a constant $K <\infty$ such that 
\[
\abs{\varphi(x')-\varphi(y')} \le K \abs{x'-y'},\quad \forall x', y' \in \bR^{d-1}.
\]
A bounded domain $\Omega\subset \bR^d$ is called a Lipschitz domain if $\partial\Omega$ locally is given by the graph of a Lipschitz function.
A domain $\Omega\subset \bR^d$ is called a Lipschitz graph domain if
\[
\Omega=\set{x=(x',x_d)\in \bR^d : x_d>\varphi(x')},
\]
where $\varphi:\bR^{d-1} \to \bR$ is a Lipschitz function.
Throughout the entire article, we let $\Omega$ be a Lipschitz domain or Lipschitz graph domain in $\bR^d$.

For $p\ge 1$ and $k$ a nonnegative integer, we denote by $W^{k,p}(\Omega)$ the usual Sobolev space.
When $\Omega$ is a Lipschitz domain, we define the space $\tilde{W}^{1,2}(\Omega)$ as the family of all functions $u\in W^{1,2}(\Omega)$ satisfying $\int_{\partial\Omega} u =0$ in the sense of trace.
We warn the reader that the space $\tilde{W}^{1,2}(\Omega)$ is different from the space $\tilde{W}_1^2(B)$ used in \cite{KP93}.
By using Rellich-Kondrachov compactness theorem, one can easily show that there is a constant $C=C(d,\Omega)$ such that
\begin{equation}
\label{eq2.2bp}
\norm{u}_{L^2(\Omega)} \le C  \norm{D u}_{L^2(\Omega)},\quad \forall u\in \tilde{W}^{1,2}(\Omega).
\end{equation}
The space $Y^{1,2}(\Omega)$ is defined as the family of all weakly differentiable functions $u\in L^{2d/(d-2)}(\Omega)$, whose weak derivatives are functions in $L^2(\Omega)$.
The space $Y^{1,2}(\Omega)$ is endowed with the norm
\[
\norm{u}_{Y^{1,2}(\Omega)}:=\norm{u}_{L^{2d/(d-2)}(\Omega)}+\norm{D u}_{L^2(\Omega)}.
\]
For a Lipschitz graph domain $\Omega$ with Lipschitz constant $K$, it is easy to show (see Appendix) that the following Sobolev inequality holds:
\begin{equation}					\label{eq5.3nh}
\norm{u}_{L^{2d/(d-2)}(\Omega)} \le C(d,K) \norm{D u}_{L^2(\Omega)},\quad \forall u\in Y^{1,2}(\Omega).
\end{equation}
We denote $\Omega_R(x)=\Omega\cap B_R(x)$ and $\varSigma_R(x)=\partial\Omega\cap B_R(x)$ for any $R>0$.
We abbreviate $\Omega_R=\Omega_R(x)$ and $\varSigma_R=\varSigma_R(x)$ if the point $x$ is well understood in the context.
We define $d_x=\dist(x,\partial\Omega)=\inf\set{\abs{x-y}:y\in\partial\Omega}$.

\subsection{Elliptic systems}
Let $L$ be an elliptic operator acting on column vector valued functions $\vec u=(u^1,\ldots,u^m)^T$ defined on a subset of $\bR^d$, in the following way:
\[
L\vec u = -D_\alpha \bigl(\vec A^{\alpha\beta}\, D_\beta \vec u\bigr),
\]
where we use the usual summation convention over repeated indices $\alpha,\beta=1,\ldots, d$, and $\vec A^{\alpha\beta}=\vec{A}^{\alpha\beta}(x)$ are $m\times m$ matrix valued functions defined on the whole space $\bR^d$ with entries $A^{\alpha\beta}_{ij}$ that satisfy the strong ellipticity condition
\begin{equation}    \label{eqP-02}
A^{\alpha\beta}_{ij}(x)\xi^j_\beta \xi^i_\alpha \ge \lambda \bigabs{\vec \xi}^2:=
\lambda \sum_{i=1}^m\sum_{\alpha=1}^d \Abs{\xi^i_\alpha}^2, \quad\forall \vec\xi \in \bR^{md},\quad\forall x\in\bR^d,
\end{equation}
for some constant $\lambda>0$ and also the uniform boundedness condition
\begin{equation}    \label{eqP-03}
\Abs{A^{\alpha\beta}_{ij}(x) \xi_\alpha^j \eta_\beta^i} \le M \bigabs{\vec \xi} \bigabs{\vec \eta},\quad \forall \vec \xi ,\vec \eta \in \bR^{md},\quad \forall x\in\bR^d,
\end{equation}
for some constant $M>0$.
Notice that the $i$-th component of the column vector $L \vec u$ coincides with $L_{ij} u^j$ in \eqref{eq0.0}.
The adjoint operator $\Lt$ is defined by
\[
\Lt \vec u = -D_\alpha ({}^t\!\vec A^{\alpha\beta} D_\beta \vec u),
\]
where ${}^t\!\vec A^{\alpha\beta}=(\vec A^{\beta\alpha})^T$; i.e., ${}^t\!A^{\alpha\beta}_{ij}=A^{\beta\alpha}_{ji}$.
\subsection{Neumann boundary value problem}			\label{sec:nbp}
We denote by $\vec A D\vec u \cdot \vec n$ the conormal derivative of $\vec u$ associated with the operator $L$; i.e., $i$-th component of $\vec A D\vec u \cdot \vec n$ is defined by
\[
(\vec A D\vec u \cdot \vec n)^i=A^{\alpha\beta}_{ij} D_\beta u^j n_\alpha,
\]
where $\vec n=(n_1,\ldots, n_d)^T$ is the outward unit normal to $\partial\Omega$.
Let $\Sigma$ be an open subset of $ \partial\Omega$ and  $\vec f  \in L^1_{loc}(\Omega)^m$ and $\vec g \in L^1_{loc}(\Sigma)^m$.
We shall say that $\vec u \in W^{1,1}_{loc}(\Omega)^m$ is a weak solution of
\[
L\vec u=\vec f \;\text{ in }\;\Omega,\quad 
\vec A D\vec u \cdot \vec n = \vec g\; \text{ on }\;\Sigma
\]
if the following identity holds:
\begin{equation}					\label{eq2.1ax}
\int_\Omega A^{\alpha\beta}_{ij} D_\beta u^j D_\alpha \phi^i -\int_\Sigma g^i \phi^i=\int_\Omega f^i \phi^i,\quad \forall \vec \phi \in C_c^\infty(\Omega \cup \Sigma)^m.
\end{equation}
Observe that \eqref{eq2.1ax} makes sense if $\vec f$ is a vector-valued measure in $\Omega$; see part ii) in the definition of Neumann function below.
We are mostly interested in the case when $\Sigma=\partial\Omega$.

\subsubsection{Neumann problem in a bounded Lipschitz domain}
Let $\Omega\subset \bR^d$ be a bounded Lipschitz domain.
Notice that the inequality \eqref{eq2.2bp} implies that $\vec H:=\tilde{W}^{1,2}(\Omega)^m$ becomes a Hilbert space with the inner product
\begin{equation}					\label{eq2.2hp}
\ip{\vec u, \vec v}_{\vec H}:=\int_{\Omega}D_\alpha u^i D_\alpha v^i.
\end{equation}
If we define the bilinear form associated to the operator $L$ as
\begin{equation}					\label{eq2.2ht}
B(\vec u, \vec v):=\int_{\Omega} A^{\alpha\beta}_{ij}D_\beta u^j D_\alpha v^i,
\end{equation}
then by \eqref{eqP-02} and \eqref{eqP-03}, the bilinear form $B$ becomes coercive and bounded on $\vec H$.
Observe that by the inequality \eqref{eq2.2bp} and the Sobolev imbedding theorem, we have 
\begin{equation}		\label{eq2.3cd}
\norm{\vec u}_{L^{2d/(d-2)}(\Omega)} \le C  \norm{D \vec u}_{L^2(\Omega)}=C\norm{\vec u}_{\vec H},\quad \forall \vec u\in \vec H=\tilde{W}^{1,2}(\Omega)^m.
\end{equation}
Let $\vec f \in L^{2d/(d+2)}(\Omega)^m$ and $\vec g \in L^2(\partial\Omega)^m$ satisfy the ``compatibility'' condition
\begin{equation}			\label{eq2.5iw}
\int_\Omega \vec f + \int_{\partial\Omega} \vec g=0.
\end{equation}
Then, by the inequality \eqref{eq2.3cd} and the trace theorem combined with \eqref{eq2.2bp}, we find that
\[
F(\vec u):=\int_\Omega \vec f\cdot \vec u+ \int_{\partial\Omega}\vec g\cdot \vec u
\]
is a bounded linear functional on $\vec H$.
Therefore, the Lax-Milgram theorem implies that there exists a unique $\vec u$ in $\vec H$ such that $B(\vec u, \vec v)=F(\vec v)$ for all $\vec v \in \vec H=\tilde{W}^{1,2}(\Omega)^m$.
Observe that any function $\vec v\in W^{1,2}(\Omega)^m$ is represented as a sum of a function in $\vec H$ and a constant vector in $\bR^m$ as follows.
\[
\vec v=\left(\vec v-\fint_{\partial\Omega} \vec v \right)+\fint_{\partial\Omega} \vec v=: \tilde{\vec v}+ \vec c.
\]
Notice that the condition \eqref{eq2.5iw} implies $F(\vec c)=0$.
Then the identity $B(\vec u,\tilde{\vec v})=F(\tilde{\vec v})$ yields
\begin{equation}			\label{eq2.6ux}
\int_\Omega A^{\alpha\beta}_{ij} D_\beta u^j D_\alpha v^i =\int_\Omega f^i v^i+\int_{\partial\Omega} g^i v^i,\quad \forall \vec v \in W^{1,2}(\Omega)^m.
\end{equation}
Therefore, we have a unique solution $\vec u$ in $\vec H=\tilde{W}^{1,2}(\Omega)^m$ of the Neumann problem
\[
\left\{
\begin{aligned}
L\vec u&=\vec f \quad\text{in }\;\Omega,\\
\vec A D\vec u \cdot \vec n&= \vec g \quad \text{on }\;\partial\Omega.
\end{aligned}
\right.
\]
provided $\vec f \in L^{2d/(d+2)}(\Omega)^m$ and $\vec g \in L^2(\partial\Omega)^m$ satisfy the compatibility condition \eqref{eq2.5iw}.

\subsubsection{Neumann problem in a Lipschitz graph domain}
Let $\Omega\subset \bR^d$ be a Lipschitz graph domain and recall the inequality \eqref{eq5.3nh}.
Similar to the bounded Lipschitz domain case,  $\vec H:=Y^{1,2}(\Omega)^m$ becomes a Hilbert space with the inner product \eqref{eq2.2hp}.
Also, the bilinear form $B$ in \eqref{eq2.2ht} is coercive and bounded in $Y^{1,2}(\Omega)^m$.
For any $\vec f \in L^{2d/(d+2)}(\Omega)^m$, the inequality \eqref{eq5.3nh} implies that
\[
F(\vec v):=\int_{\Omega} f^i v^i
\]
is a bounded linear functional on $Y^{1,2}(\Omega)^m$.
Therefore, the Lax-Milgram theorem implies that there exists a unique $\vec u$ in $Y^{1,2}(\Omega)^m$ such that $B(\vec u, \vec v)=F(\vec v)$ for all $\vec v \in Y^{1,2}(\Omega)^m$; i.e., we have
\begin{equation}					\label{eq5.4aa}
\int_\Omega A^{\alpha\beta}_{ij} D_\beta u^j D_\alpha v^i =\int_\Omega f^i v^i,\quad \forall \vec v \in Y^{1,2}(\Omega)^m.
\end{equation}
Hereafter, we shall say that $\vec u$ is a unique solution in $Y^{1,2}(\Omega)^m$ of the Neumann problem
\[
\left\{
\begin{aligned}
L\vec u&=\vec f \quad\text{in }\;\Omega,\\
\vec A D\vec u \cdot \vec n&= 0 \quad \text{on }\;\partial\Omega
\end{aligned}
\right.
\]
if $\vec u\in Y^{1,2}(\Omega)^m$ and satisfies the identity \eqref{eq5.4aa}.

\subsection{Neumann function}			
In the definitions below, $\vec N=\vec N(x,y)$ will be an $m\times m$ matrix valued function with measurable entries $N_{ij}:\Omega\times\Omega\to \overline\bR$.

\subsubsection{Neumann function in a bounded Lipschitz domain}		\label{sec:nf}
We say that $\vec N$ is a Neumann function of $L$ in a bounded Lipschitz domain $\Omega$ if it satisfies the following properties:
\begin{enumerate}[i)]
\item
$\vec N(\cdot,y) \in W^{1,1}_{loc}(\Omega)$ and $\vec N(\cdot,y) \in W^{1,2}(\Omega\setminus B_r(y))$ for all $y\in\Omega$ and $r>0$.
Moreover, $\int_{\partial\Omega} \vec N(\cdot,y)=0$ in the sense of trace.
\item
 $L\vec N(\cdot,y)=\delta_y \vec I$ in $\Omega$ and $\vec A D \vec N(\cdot,y)\cdot \vec n= -\frac{1}{\abs{\partial\Omega}} \vec I$ on $\partial\Omega$ for all $y\in\Omega$ in the  sense
\begin{equation}        \label{eq2.6r}
\int_{\Omega}A^{\alpha\beta}_{ij} D_\beta N_{jk}(\cdot,y) D_\alpha \phi^i+ \frac{1}{\abs{\partial\Omega}}\int_{\partial\Omega} \phi^k = \phi^k(y),\quad
\forall \vec \phi\in C^\infty(\overline\Omega)^m.
\end{equation}
\item
For any $\vec f=(f^1,\ldots, f^m)^T \in C_c^\infty(\Omega)^m$, the function $\vec u$ given by
\begin{equation}        \label{eq2.9x}
\vec u(x):=\int_\Omega \vec N(y,x)^T \vec f(y)\,dy
\end{equation}
is a unique solution in $\tilde{W}^{1,2}(\Omega)^m$ of the problem
\begin{equation}				\label{eq2.10yq}
\left\{
\begin{aligned}
\Lt\vec u&=\vec f \quad\text{in }\;\Omega,\\
{}^t\!\vec A D\vec u \cdot \vec n&= -\frac{1}{\abs{\partial\Omega}} \int_\Omega \vec f \quad \text{on }\;\partial\Omega.
\end{aligned}
\right.
\end{equation}
\end{enumerate}

\subsubsection{Neumann function in a Lipschitz graph domain}		\label{sec:nf2}
We say that $\vec N$ is a Neumann function of $L$ in a Lipschitz graph domain $\Omega$ if it satisfies the following properties:
\begin{enumerate}[i)]
\item
$\vec N(\cdot,y) \in W^{1,1}_{loc}(\Omega)$ and $\vec N(\cdot,y) \in Y^{1,2}(\Omega\setminus B_r(y))$ for all $y\in\Omega$ and $r>0$.
\item
 $L\vec N(\cdot,y)=\delta_y \vec I$ in $\Omega$ and $\vec A D \vec N(\cdot,y)\cdot \vec n= 0$ on $\partial\Omega$ for all $y\in\Omega$ in the  sense
\begin{equation}        \label{eq5.6r}
\int_{\Omega}A^{\alpha\beta}_{ij} D_\beta N_{jk}(\cdot,y) D_\alpha \phi^i = \phi^k(y),\quad \forall \vec \phi=(\phi^1\ldots,\phi^m)^T\in C_c^\infty(\overline\Omega)^m.
\end{equation}
\item
For any $\vec f\in C_c^\infty(\Omega)^m$, the function $\vec u$ given by \eqref{eq2.9x} is a unique solution in $Y^{1,2}(\Omega)^m$ of the problem
\begin{equation}				\label{eq5.10yq}
\left\{
\begin{aligned}
\Lt\vec u&=\vec f \quad\text{in }\;\Omega,\\
{}^t\!\vec A D\vec u \cdot \vec n&= 0\quad \text{on }\;\partial\Omega.
\end{aligned}
\right.
\end{equation}
\end{enumerate}

We point out that part iii) in the above definitions give the uniqueness of a Neumann function.
Indeed, let $\tilde{\vec N}(x,y)$ is another function satisfying the above properties.
Then by the uniqueness, we have
\[
\int_\Omega (\vec N -\tilde{\vec N})(y,x)^T \vec f(y)\,dy=0,\quad \forall \vec f \in C_c^\infty(\Omega)^m,
\]
and thus we conclude that $\vec N = \tilde{\vec N}$ a.e. in $\Omega\times\Omega$.

\mysection{Main results} \label{main}
The following ``interior H\"older continuity'' condition $\IH$ means that weak solutions of $L\vec u=0$ and $\Lt \vec u=0$ enjoy interior H\"older continuity.
In the case $m=1$, it is a consequence of the celebrated De Giorgi-Moser-Nash theorem.
If $m>1$ and $d>2$, it is not true in general, but however, if the coefficients of the system \eqref{eq0.0} belong to the class of VMO and if $\Omega$ is bounded, then it is known that the condition \eqref{IH} holds in that case; see e.g., \cite[Lemma~5.3]{HK07}.

\begin{CIH}
There exist $\mu_0\in (0,1]$ and $C_0>0$ such that for all $x\in\Omega$ and $0<R < d_x$, where $d_x=\dist(x,\partial\Omega)$, the following holds:
If $\vec u\in W^{1,2}(B_R(x))$ is a weak solution of either $L\vec u=0$ or $\Lt \vec u=0$ in $B_R=B_R(x)$, then $\vec u$ is H\"older continuous in $B_R$ with the following estimate:
\begin{equation}					\tag{IH}\label{IH}
[\vec u]_{C^{\mu_0}(B_{R/2})} \le C_0 R^{-\mu_0}\left(\fint_{B_R} \abs{\vec u}^2\right)^{1/2},
\end{equation}
where $[\vec u]_{C^{\mu_0}(B_{R/2})}$ denotes the usual H\"older seminorm.
\end{CIH}

\begin{theorem}	\label{thm1}
Let $\Omega\subset \bR^d$ ($d\ge 3$) be a bounded Lipschitz domain.
Assume the condition $\IH$.
Then there exist Neumann functions $\vec N(x,y)$ of $L$ and $\tilde{\vec N}(x,y)$ of $\Lt$ in $\Omega$.
We have $\vec N(\cdot,y),\, \tilde{\vec N}(\cdot,y) \in C^{\mu_0}_{loc}(\Omega\setminus \set{y})$ for all $y\in\Omega$ and the following the identity holds:
\begin{equation}					\label{eq3.01mq}
\tilde{\vec N}(x,y):=\vec N(y,x)^T,\quad \forall x, y \in \Omega,\;\; x\ne y.
\end{equation}
Moreover, for any $\vec f \in L^q(\Omega)^m$ with $q>d/2$ and $\vec g\in L^2(\partial\Omega)^m$ satisfying $\int_\Omega \vec f+\int_{\partial\Omega} \vec g=0$, the function $\vec u$ given by
\begin{equation}        \label{eqM1.e}
\vec u(x):=\int_\Omega \vec N(x,y) \vec f(y) \,dy+ \int_{\partial\Omega} \vec N(x,y) \vec g(y)\,d\sigma(y)
\end{equation}
is a unique solution in $\tilde{W}^{1,2}(\Omega)^m$ of the problem
\begin{equation}			\label{eq3.4rm}  
\left\{
\begin{aligned}
L\vec u&=\vec f \quad\text{in }\;\Omega,\\
\vec A D\vec u \cdot \vec n&= \vec g\quad \text{on }\;\partial\Omega.
\end{aligned}
\right.
\end{equation}
Furthermore, the following estimates hold for all $y\in\Omega$:
\begin{enumerate}[i)]
\item
$\norm{\vec N(\cdot,y)}_{L^{2d/(d-2)}(\Omega\setminus B_r(y))}+\norm{D\vec N(\cdot,y)}_{L^2(\Omega\setminus B_r(y))} \le C r^{1-d/2}\;$ for all $r \in (0,d_y)$.
\item
$\norm{\vec N(\cdot,y)}_{L^p(B_r(y))}\le C r^{2-d+d/p}\;$ for all $r\in (0,d_y)$, where $p\in [1,\frac{d}{d-2})$.
\item
$\abs{\set{x\in\Omega:\abs{\vec N(x,y)}>t}}\le C t^{-d/(d-2)}\;$ for all $t> d_y^{2-d}$.
\item
$\norm{D\vec N(\cdot,y)}_{L^p(B_r(y))}\le C r^{1-d+d/p}\;$ for all $r\in (0, d_y)$, where $p\in [1,\frac{d}{d-1})$.
\item
$\abs{\set{x\in\Omega:\abs{D_x \vec N(x,y)}>t}}\le C t^{-d/(d-1)}\;$ for all $t> d_y^{1-d}$.
\item
$\abs{\vec N(x,y)}\le C \abs{x-y}^{2-d}\;$ whenever $0<\abs{x-y}<d_y/2$.
\item
$\abs{\vec N(x,y)-\vec N(x',y)} \le C \abs{x-x'}^{\mu_0} \abs{x-y}^{2-d-\mu_0}\;$ if  $x\ne y$ and  $2\abs{x-x'}<\abs{x-y}<d_y/2$.
\end{enumerate}
In the above, $C=C(d,m,\lambda,M,\Omega,\mu_0,C_0)>0$ and $C$ depends on $p$ as well in ii) and  iv).
The estimates i) -- vii) are also valid for $\tilde{\vec N}(x,y)$.
\end{theorem}

\begin{remark}
Observe that if $\vec f \in L^q(\Omega)^m$, where $q>d/2$, satisfies $\int_\Omega \vec f=0$, then we may take $\vec g=0$ in \eqref{eqM1.e} and conclude that
\[
\vec u(x):=\int_\Omega \vec N(x,y) \vec f(y)\,dy
\]
is a unique solution in $\tilde{W}^{1,2}(\Omega)^m$ of the problem
\[
\left\{
\begin{aligned}
L\vec u&=\vec f \quad\text{in }\;\Omega,\\
\vec A D\vec u \cdot \vec n&= 0 \quad \text{on }\;\partial\Omega.
\end{aligned}
\right.
\]
Similarly, if $\vec g\in L^2(\partial\Omega)^m$ satisfies $\int_{\partial\Omega} \vec g =0$, then we find that
\[
\vec u(x):=\int_{\partial\Omega} \vec N(x,y) \vec g(y)\,d\sigma(y)
\]
is a unique solution in $\tilde{W}^{1,2}(\Omega)^m$ of the problem
\[
\left\{
\begin{aligned}
L\vec u&= 0 \quad\text{in }\;\Omega,\\
\vec A D\vec u \cdot \vec n&= \vec g \quad \text{on }\;\partial\Omega.
\end{aligned}
\right.
\]
Also, the following estimates are easy consequences of the identity \eqref{eq3.01mq} and the estimates i) -- vii) for $\tilde{\vec N}(\cdot, x)$:
\begin{enumerate}[i)]
\em
\item
$\norm{\vec N(x,\cdot)}_{L^{2d/(d-2)}(\Omega\setminus B_r(x))}+\norm{D\vec N(x,\cdot)}_{L^2(\Omega\setminus B_r(x))} \le C r^{1-d/2}\;$ for all $r \in (0,d_x)$.
\item
$\norm{\vec N(x,\cdot)}_{L^p(B_r(x))}\le C r^{2-d+d/p}\;$ for all $r\in (0,d_x)$, where $p\in [1,\frac{d}{d-2})$.
\item
$\abs{\set{y\in\Omega:\abs{\vec N(x,y)}>t}}\le C t^{-d/(d-2)}\;$ for all $t> d_x^{2-d}$.
\item
$\norm{D\vec N(x,\cdot)}_{L^p(B_r(x))}\le C r^{1-d+d/p}\;$ for all $r\in (0, d_x)$, where $p\in [1,\frac{d}{d-1})$.
\item
$\abs{\set{y\in\Omega:\abs{D_y \vec N(x,y)}>t}}\le C t^{-d/(d-1)}\;$ for all $t> d_x^{1-d}$.
\item
$\abs{\vec N(x,y)}\le C \abs{x-y}^{2-d}\;$ whenever $0<\abs{x-y}<d_x/2$.
\item
$\abs{\vec N(x,y)-\vec N(x,y')} \le C \abs{y-y'}^{\mu_0} \abs{x-y}^{2-d-\mu_0}\;$ if $x\ne y$ and $2\abs{y-y'}<\abs{x-y}<d_x/2$.
\end{enumerate}
In particular, we have $\abs{\vec N(x,y)}\leq C \abs{x-y}^{2-d}$ whenever $0<\abs{x-y}<\frac{1}{2}\max(d_x,d_y)$.

\end{remark}

The following ``local boundedness'' condition \eqref{LB} is used to obtain global pointwise bounds for the Neumann function $\vec N(x,y)$ of $L$ in $\Omega$.
Again, in the case $m=1$, it is well known that the condition \eqref{LB} holds in bounded Lipschitz domains; see e.g., \cite{Miranda}.
In the case when $m>1$, this condition does not hold in general and requires certain restrictions on the coefficients and domains.
It can be shown, for example, that if the coefficients belong to the VMO class and the domain is bounded and has $C^1$ boundary, then the condition \eqref{LB} holds via $W^{1,p}$ estimates;
see Appendix.

\begin{CLB}
There exists a constant $C_1>0$ such that the following holds:
For any $\vec f \in C_c^\infty(\Omega)^m$ and $\vec g \in C^\infty(\partial\Omega)$ satisfying $\int_\Omega \vec f + \int_{\partial\Omega} \vec g =0$, let $\vec u \in \tilde{W}^{1,2}(\Omega)^m$ be a unique weak solution of the problem
\[
\begin{aligned}
\left\{
\begin{aligned}
L\vec u&=\vec f \;\text{ in }\;\Omega\\
\vec A D\vec u \cdot \vec n&=  \vec g \; \text{ on }\;\partial\Omega
\end{aligned}
\right. 
\qquad\;\text{ or }\;\qquad
\left\{
\begin{aligned}
\Lt\vec u&=\vec f \quad\text{in }\;\Omega\\
{}^t\!\vec A D\vec u \cdot \vec n&=   \vec g \quad \text{on }\;\;\partial\Omega.
\end{aligned}
\right.
\end{aligned}
\]
Then for all $x\in\Omega$ and $0<R < \diam(\Omega)$, we have
\begin{equation*}			\tag{LB}\label{LB}
\norm{\vec u}_{L^\infty(\Omega_{R/2}(x))} \le C_1 \left(R^{-d/2} \norm{\vec u}_{L^2(\Omega_R(x))}+ R^2 \norm{\vec f}_{L^\infty(\Omega_R(x))}+R \norm{\vec g}_{L^\infty(\Sigma_R(x))}\right).
\end{equation*}
\end{CLB}

\begin{theorem}							\label{thm2}
Let $\Omega\subset \bR^d$ ($d\ge 3$) be a bounded Lipschitz domain.
Assume the condition \eqref{IH} and let $\vec N(x,y)$ be the Neumann function of $L$ in $\Omega$ as constructed in Theorem~\ref{thm1}.
If we further assume the condition \eqref{LB}, then we have the following global pointwise bound for the Neumann function:
\begin{equation}							\label{eq2.17dc}
\abs{\vec N(x,y)} \le C \abs{x-y}^{2-d}\quad \text{for all }\, x,y\in\Omega\;\text{ with }\;x\ne y,
\end{equation}
where $C=C(d,m,\lambda, M, \Omega, C_1)$.
Moreover, for all $y\in\Omega$ and $0<r<\diam(\Omega)$, we have
\begin{enumerate}[i)]
\item
$\norm{\vec N(\cdot,y)}_{L^{2d/(d-2)}(\Omega\setminus B_r(y))}+\norm{D\vec N(\cdot,y)}_{L^2(\Omega\setminus B_r(y))} \le C r^{1-d/2}$.
\item
$\norm{\vec N(\cdot,y)}_{L^p(B_r(y))}\le C r^{2-d+d/p}\;$ for $p\in [1,\frac{d}{d-2})$.
\item
$\abs{\set{x\in\Omega:\abs{\vec N(x,y)}>t}}\le C t^{-d/(d-2)}\;$ for all $t> 0$.
\item
$\norm{D\vec N(\cdot,y)}_{L^p(B_r(y))}\le C r^{1-d+d/p}\;$ for $p\in [1,\frac{d}{d-1})$.
\item
$\abs{\set{x\in\Omega:\abs{D_x \vec N(x,y)}>t}}\le C t^{-d/(d-1)}\;$ for all $t> 0$.
\end{enumerate}
In the above, $C=C(d,m,\lambda,M,\Omega,C_1)>0$ and $C$ depends on $p$ as well in ii) and  iv).
The estimates i) -- v) are also valid for the Neumann function $\tilde{\vec N}$ of the adjoint $\Lt$.
\end{theorem}

\begin{remark}					\label{rmk3.3}
As we have pointed out, the condition \eqref{LB} is satisfied, for example, in the scalar case and also in the case when the system has VMO coefficients and the domain is of class $C^1$.
In fact, in those cases, we also have the following ``local H\"older continuity'' condition:
There exist constants $\mu_0\in (0,1]$ and $C_1>0$ such that for all $x\in\Omega$ and $0<R<\diam(\Omega)$, the following holds:
Let $\vec u \in W^{1,2}(\Omega_R(x))^m$ be a  weak solution of either 
\begin{align*}
L\vec u&=0 \;\text{ in }\; \Omega_R(x),\quad \vec A D\vec u \cdot \vec n=\vec g \;\text{ on }\; \Sigma_R(x)\\
\text{or}\quad \Lt \vec u&=0 \;\text{ in }\; \Omega_R(x),\quad {}^t\!\vec A D\vec u \cdot \vec n=\vec g \;\text{ on }\; \Sigma_R(x), 
\end{align*}
where $\vec g\in C^\infty(\partial\Omega)^m$, then we have
\begin{equation*}			\tag{LH}\label{LH}
R^{\mu_0} [\vec u]_{C^{\mu_0}(\Omega_{R/2}(x))} \le C_1 \left( R^{-d/2} \norm{\vec u}_{L^2(\Omega_R(x))}+ R\norm{\vec g}_{L^\infty(\Sigma_R(x))}\right).
\end{equation*}
By using \eqref{LH} and modifying the proof for the estimate vii) in Theorem~\ref{thm1} (c.f. the proof for \eqref{eq2.17dc} in Section~\ref{sec4.2}), we have the following global version of the estimate vii):
\[
\abs{\vec N(x,y)-\vec N(x',y)} \le C \abs{x-x'}^{\mu_0} \abs{x-y}^{2-d-\mu_0}\;\text{ if }\;x\ne y\;\text{ and }\;2\abs{x-x'}<\abs{x-y},
\]
where $C=C(d,m,\lambda,M,\Omega,\mu_0, C_1)>0$.
The same estimate is also valid for $\tilde{\vec N}$.
\end{remark}

Finally, the following theorem says that the converse of Theorem~\ref{thm2} is also true, and thus that condition $\LB$ is equivalent to a global bound \eqref{eq2.17dc} for the Neumann function.
\begin{theorem}         \label{thm3}
Let $\Omega\subset \bR^d$ ($d\ge 3$) be a bounded Lipschitz domain.
Assume the condition \eqref{IH} and let $\vec N(x,y)$ be the Neumann function of  $L$ in $\Omega$.
Suppose there exists a constant $C_2$ such that  we have
\begin{equation}							\label{eq2.17dd}
\abs{\vec N(x,y)} \le C_2 \abs{x-y}^{2-d},\quad \forall x,y\in \Omega,\quad x\ne y.
\end{equation}
Then the condition \eqref{LB} is satisfied in $\Omega$ with $C_1=C_1(d,m,\lambda,M, \Omega, C_2)$.
\end{theorem}

\mysection{Proofs of main theorems}					\label{pf}
\subsection{Proof of Theorem~\ref{thm1}}
We closely follow the proof of \cite[Theorem~3.1]{HK07}.
Let us fix a function $\Phi \in C_c^\infty(\bR^d)$ such that $\Phi$ is supported in $B_1(0)$, $0\le \Phi \le 2$, and $\int_{\bR^d} \Phi=1$.
Let $y\in \Omega$ be fixed but arbitrary.
For $\epsilon>0$, we define
\[
\Phi_\epsilon(x)=\epsilon^{-d}\Phi((x-y)/\epsilon).
\]
Let $\vec v=\vec v_{\epsilon, y, k}$ be a unique weak solution in $\tilde{W}^{1,2}(\Omega)^m$ of the problem (see Section~\ref{sec:nbp})
\begin{equation}					\label{eq4.0ik}
\left\{
\begin{aligned}
L \vec v &= \Phi_\epsilon \vec e_k\;\text{ in }\;\Omega,\\
\vec A D\vec v \cdot \vec n &=-(1/\abs{\partial\Omega})\vec e_k \;\text{ on }\; \partial\Omega,
\end{aligned}
\right.
\end{equation}
where $\vec e_k$ is the $k$-th unit vector in $\bR^m$.
We define the ``mollified Neumann function'' $\vec N^\epsilon(\cdot,y)=(N^\epsilon_{jk}(\cdot,y))_{j,k=1}^m$ by
\begin{equation}					\label{eq4.1uu}
N^\epsilon_{jk}(\cdot,y)=v^j=v^j_{\epsilon,y,k}.
\end{equation}

Then $\vec N^\epsilon(\cdot,y)$ satisfies the following identity (see \eqref{eq2.6ux}):
\begin{equation}				\label{eq4.1tv}
\int_{\Omega}A^{\alpha\beta}_{ij} D_\beta N_{jk}^\epsilon(\cdot,y) D_\alpha \phi^i+ \frac{1}{\abs{\partial\Omega}}\int_{\partial\Omega} \phi^k = \int_ {\Omega_\epsilon(y)} \Phi_\epsilon \phi^k ,\quad
\forall \vec \phi \in W^{1,2}(\Omega)^m.
\end{equation}
By the definition of the space $\tilde{W}^{1,2}(\Omega)$, we have in particular the following identity:
\begin{equation}				\label{eq4.2xa}
\int_{\Omega}A^{\alpha\beta}_{ij} D_\beta N_{jk}^\epsilon(\cdot,y) D_\alpha \phi^i = \int_ {\Omega_\epsilon(y)} \Phi_\epsilon \phi^k, \quad \forall \vec \phi \in \tilde{W}^{1,2}(\Omega)^m.
\end{equation}
By taking $\vec \phi=\vec v$ in \eqref{eq4.2xa} and then using \eqref{eqP-02}, H\"older's inequality, and \eqref{eq2.3cd}, we get
\[
\lambda \norm{D\vec v}_{L^2(\Omega)}^2 \le \Abs{\int_{\Omega_\epsilon(y)} \Phi_\epsilon v^k }\le C \epsilon^{(2-d)/2}\norm{\vec v}_{L^{2d/(d-2)}(\Omega)} \le C \epsilon^{(2-d)/2} \norm{D\vec v}_{L^2(\Omega)}.
\]
Therefore, we have (recall $\vec v$ is the $k$-th column of $\vec N^\epsilon(\cdot,y)$)
\begin{equation}		\label{eqG-02}
\norm{D\vec N^\epsilon(\cdot,y)}_{L^2(\Omega)}\le C \epsilon^{(2-d)/2},\quad\text{where }\;C=C(d,m,\lambda,M,\Omega).
\end{equation}

Let $R\in (0,d_y)$ be arbitrary, but fixed.
Assume that $\vec f \in C_c^\infty(\Omega)^m$ is supported in $B_R=B_R(y) \subset \Omega$.
Let $\vec u$ be a unique weak solution in $\tilde{W}^{1,2}(\Omega)^m$ of the problem \eqref{eq2.10yq}.
We then have the following identity (recall ${}^t\!A^{\beta\alpha}_{ji} = A^{\alpha\beta}_{ij}$):
\begin{equation}				\label{eq4.2ym}
\int_\Omega A^{\alpha\beta}_{ij} D_\beta w^j D_\alpha u^i =\int_\Omega f^i w^i,\quad \forall \vec w \in \tilde{W}^{1,2}(\Omega)^m.
\end{equation}
Then by setting $\vec \phi=\vec u$ in \eqref{eq4.2xa} and setting $\vec w=\vec v_{\epsilon,y,k}$ in \eqref{eq4.2ym}, we get
\begin{equation}				\label{eq4.3wx}
\int_\Omega N^\epsilon_{ik}(\cdot,y)f^i = \int_ {\Omega_\epsilon(y)} \Phi_\epsilon u^k.
\end{equation}
Also, by taking $\vec w=\vec u$ in \eqref{eq4.2ym}, and using \eqref{eqP-02}, \eqref{eq2.3cd}, and H\"older's inequality, we get 
\[
\lambda \norm{D\vec u}_{L^2(\Omega)}^2 \le \int_\Omega A^{\alpha\beta}_{ij} D_\beta u^j D_\alpha u^i =\int_\Omega f^i u^i \le C \norm{\vec f}_{L^{2d/(d+2)}(\Omega)}\norm{D\vec u}_{L^2(\Omega)}.
\]
Therefore, we have the estimate
\begin{equation}					\label{eq4.4ur}
\norm{D\vec u}_{L^2(\Omega)} \le  C \norm{\vec f}_{L^{2d/(d+2)}(\Omega)}.
\end{equation}

We remark that the condition $\IH$ is equivalent to the property $\mathrm{(H)}$ in \cite[Definition~2.1]{HK07}; see \cite[Lemma~2.3 and 2.4]{HK07}.
By utilizing the condition $\IH$ and \eqref{eq4.4ur}, and following literally the same steps used in deriving \cite[Eq.~(3.15)]{HK07}, we obtain
\[
\norm{\vec u}_{L^\infty(B_{R/2})} \le C R^2 \norm{\vec f}_{L^\infty(B_R)},
\]
where $C=C(d,m,\lambda,M, \Omega,\mu_0, C_0)$.
Since $R\in(0,d_y)$ is arbitrary, we get from the above estimate and \eqref{eq4.3wx} that
\[
\Biggabs{\int_{B_R} N^\epsilon_{ik}(\cdot,y) f^i} \le C R^{2} \norm{\vec f}_{L^\infty(B_R)},\quad \forall \vec f\in C_c^\infty(B_R),\;\; \forall \epsilon \in (0,R/2),\;\;  \forall R\in(0,d_y).
\]
Therefore, by duality, we conclude that
\[
\norm{\vec N^\epsilon (\cdot,y)}_{L^1(B_R(y))}\le C R^{2},\quad \forall \epsilon \in(0,R/2),\;\;  \forall R\in(0,d_y).
\]

Now, for any $x\in\Omega$ such that $0<\abs{x-y}<d_y/2$, let us take $R:=2\abs{x-y}/3$.
Notice that if $\epsilon<R/2$, then $\vec N^\epsilon(\cdot,y)\in W^{1,2}(B_R(x))^{m^2}$ and satisfies $L\vec  N^\epsilon(\cdot,y)= 0$ in $B_R(x)$.
Then by following the same line of argument used in deriving \cite[Eq.~(3.19)]{HK07},  for any $x, y\in\Omega$ satisfying $0<\abs{x-y}<d_y/2$, we have
\begin{equation}        \label{eq4.5ac}
\abs{\vec N^\epsilon(x,y)} \le C \abs{x-y}^{2-d},\quad \forall \epsilon <\abs{x-y}/3.
\end{equation}

Next, fix any $r\in (0,d_y/2)$ and let $\vec v_\epsilon$ be the $k$-th column of $\vec N^\epsilon(\cdot,y)$, where $k=1,\ldots, m$ and $0<\epsilon<r/6$.
Let $\eta$ be a smooth function on $\bR^d$ satisfying
\begin{equation}        \label{eq4.19h}
0\le \eta\le 1,\quad \eta\equiv 1\,\text{ on }\,\bR^d\setminus B_r(y),\quad \eta\equiv 0\,\text{ on }\, B_{r/2}(y),\quad\text{and}\quad \abs{D\eta} \le 4/r.
\end{equation}
We set $\vec \phi=\eta^2\vec v_\epsilon$ in \eqref{eq4.2xa} and then use \eqref{eq4.5ac} to obtain 
 \begin{equation}       \label{eq4.24y}
\int_\Omega \eta^2 \abs{D\vec v_\epsilon}^2 \le C \int_\Omega \abs{D\eta}^2 \abs{\vec v_\epsilon}^2 \le C r^{-2} \int_{B_r(y)\setminus B_{r/2}(y)} \abs{x-y}^{2(2-d)}\,dx\le  C r^{2-d}.
\end{equation}
Therefore, by \eqref{eq2.3cd} and \eqref{eq4.24y}, we obtain
\[
\norm{\vec v_\epsilon}_{L^{2d/(d-2)}(\Omega\setminus B_r(y))} \le \norm{\eta \vec v_\epsilon}_{L^{2d/(d-2)}(\Omega)}\le C \norm{D (\eta \vec v_\epsilon)}_{L^2(\Omega)} \le C r^{(2-d)/2}
\]
provided that $0<\epsilon<r/6$.
On the other hand, if $\epsilon\ge r/6$, then \eqref{eqG-02} implies
\[
\norm{\vec v_\epsilon}_{L^{2d/(d-2)}(\Omega\setminus B_{r}(y))} \le \norm{\vec v_\epsilon}_{L^{2d/(d-2)}(\Omega)} \le C \norm{D \vec v_\epsilon}_{L^2(\Omega)} \le C r^{(2-d)/2}.
\]
By combining the above two estimates, we obtain
\begin{equation}							\label{eqG-20}
\norm{\vec N^\epsilon(\cdot,y)}_{L^{2d/(d-2)}(\Omega\setminus B_r(y))} \le C r^{(2-d)/2}, \quad \forall r \in (0,d_y/2),\quad \forall \epsilon>0.
\end{equation}
Notice from \eqref{eq4.24y} and \eqref{eq4.19h} that for $0<\epsilon<r/6$, we have
\[
\norm{D\vec N^\epsilon(\cdot,y)}_{L^2(\Omega\setminus B_r(y))}\le C r^{(2-d)/2}.
\]
In the case when $\epsilon \ge r/6$, we obtain from \eqref{eqG-02} that
\[
\norm{D\vec N^\epsilon(\cdot,y)}_{L^2(\Omega\setminus B_r(y))} \le \norm{D\vec N^\epsilon(\cdot,y)}_{L^2(\Omega)} \le C \epsilon^{(2-d)/2} \le C r^{(2-d)/2}.
\]
By combining the above two inequalities, we obtain
\begin{equation}							\label{eqG-14}
\norm{D\vec N^\epsilon(\cdot,y)}_{L^2(\Omega\setminus B_r(y))}\le C r^{(2-d)/2},\quad \forall r\in(0,d_y/2),\quad \forall \epsilon>0.
\end{equation}
From the the obvious fact that $d_y/2$ and $d_y$ are comparable to each other, we find by \eqref{eqG-20} and \eqref{eqG-14} that
\begin{equation}					\label{eq4.11bs}
\norm{\vec N^\epsilon(\cdot,y)}_{L^{2d/(d-2)}(\Omega\setminus B_r(y))} +\norm{D\vec N^\epsilon(\cdot,y)}_{L^2(\Omega\setminus B_r(y))}\le C r^{(2-d)/2},\;\; \forall r\in(0,d_y),\;\; \forall \epsilon>0.
\end{equation}
From \eqref{eq4.11bs} it follows that (see \cite[pp. 147--148]{HK07})
\begin{align}
\label{eq4.12eq}
\abs{\set{x\in\Omega:\abs{\vec N^\epsilon(x,y)}>t}}\le C t^{-d/(d-2)},\quad  \forall t> d_y^{2-d},\;\; \forall \epsilon>0,\\
\label{eq4.13tz}
\abs{\set{x\in\Omega:\abs{D_x \vec N^\epsilon(x,y)}>t}} \le C t^{-d/(d-1)},\quad \forall t> d_y^{1-d},\;\; \forall \epsilon>0.
\end{align}
It is routine to derive the following strong type estimates from the above weak type estimates \eqref{eq4.12eq} and \eqref{eq4.13tz} (see e.g. \cite[p.~148]{HK07}):
\begin{align}
\label{eq4.14mq}
\norm{\vec N^\epsilon(\cdot,y)}_{L^p(B_r(y))}\le C r^{2-d+d/p},\;\; \forall r\in (0,d_y), \;\; \forall \epsilon>0,\;\text{ for }\;1\le p <d/(d-2),\\
\label{eq4.15tw}
\norm{D\vec N^\epsilon(\cdot,y)}_{L^p(B_r(y))}\le C r^{1-d+d/p},\;\; \forall r\in (0,d_y),\;\; \forall \epsilon>0,\;\text{ for }\; 1 \le p < d/(d-1),
\end{align}
where $C=C(d,m,\lambda,M,\Omega,\mu_0, C_0, p)$.

From \eqref{eqG-14}, \eqref{eq4.14mq}, and \eqref{eq4.15tw}, it follows that that there exists a sequence $\{\epsilon_\mu\}_{\mu=1}^\infty$ tending to zero and a function $\vec N(\cdot,y)$ such that $\vec N^{\epsilon_\mu}(\cdot,y) \rightharpoonup \vec N(\cdot,y)$ weakly in $W^{1,p}(B_r(y))$ for $1<p<d/(d-1)$ and all $r\in(0,d_y)$ and also that $\vec N^{\epsilon_\mu}(\cdot,y) \rightharpoonup \vec N(\cdot,y)$ weakly in $W^{1,2}(\Omega\setminus B_r(y))$ for all $r\in(0,d_y)$; see \cite[p.~159]{HK07} for the details.
Then it is routine to check that $\vec N(\cdot,y)$ satisfies the properties i) and ii) in Section~\ref{sec:nf}, and also the estimates i) -- v) in the theorem; see \cite[Section~4.1]{HK07}.

We now turn to pointwise bound for $\vec N(x,y)$. 
For any $x\in\Omega$ such that $0<\abs{x-y}<d_y/2$, set $R:=2\abs{x-y}/3$.
Notice that \eqref{eq4.11bs} implies that $\vec N(\cdot,y)\in W^{1,2}(B_R(x))$ and satisfies $L\vec N(\cdot,y)= 0$ weakly in $B_R(x)$.
Then, by \cite[Lemma~2.4]{HK07} and the estimate ii) in the theorem, we have
\[
\abs{\vec N(x,y)} \le C R^{-d} \norm{\vec N(\cdot,y)}_{L^1(B_R(x))} \le C R^{-d} \norm{\vec N(\cdot,y)}_{L^1(B_{3R}(y))} \le C R^{2-d}\le C \abs{x-y}^{2-d}.
\]
We have thus shown that the estimate vi) in the theorem holds.
Then, it is routine to see that the estimate vii) in the theorem follows from the condition $\IH$ and the above estimate.

Next, let $x\in \Omega\setminus\set{y}$ be fixed but arbitrary, and let $\tilde{\vec N}{}^{\epsilon'}(\cdot,x)\in \tilde{W}^{1,2}(\Omega)^{m^2}$ be the mollified Neumann function of the adjoint operator $\Lt$ in $\Omega$, where $\epsilon'>0$.
By setting $\vec \phi$ in \eqref{eq4.2xa} to be the $l$-th column of $\tilde{\vec N}{}^{\epsilon'}(\cdot,x)$ and utilizing an integral identity for $\tilde{\vec N}{}^{\epsilon'}(\cdot,x)$ similar to \eqref{eq4.2xa}, we obtain the following identity:
\[
\int_ {\Omega_{\epsilon'}(x)} \Phi_{\epsilon'} N_{lk}^\epsilon(\cdot,y)= \int_ {\Omega_\epsilon(y)} \Phi_\epsilon \tilde{N}{}_{kl}^{\epsilon'}(\cdot,x).
\]
Let $\tilde{\vec N}(\cdot,x)$ be a Neumann function of $\Lt$ in $\Omega$ obtained by a sequence $\set{\epsilon_\nu'}_{\nu=1}^\infty$ tending to $0$.
Then, by following the same steps as in \cite[p.~151]{HK07}, we conclude
\[
N_{lk}(x,y)=\tilde{N}_{kl}(y,x), \quad \forall k,l=1,\ldots,m,
\]
which obviously implies the identity \eqref{eq3.01mq}.
In fact, by following a similar line of reasoning as in \cite[p.~151]{HK07}, we find
\begin{align}
\vec N^\epsilon(x,y)=\epsilon^{-d}\int_{\Omega}\Phi\left(\frac{z-y}{\epsilon}\right) \vec N(x,z)\,dz, \\
\label{eq4.17ht}
\lim_{\epsilon \to 0} \vec N^\epsilon(x,y)=\vec N(x,y), \quad \forall x,y\in\Omega,\;\; x\ne y.
\end{align}

Now, let $\vec u$ be a unique solution in $\tilde{W}^{1,2}(\Omega)^m$ of the problem \eqref{eq2.10yq} with $\vec f\in C_c^\infty(\Omega)^m$.
We remark that the condition $\IH$ implies that $\vec u$ is continuous in $\Omega$; see \cite[Eq.~(3.14)]{HK07}.
By setting $\vec w$ to be the $k$-th column of $\vec N^\epsilon(\cdot,y)$ in \eqref{eq4.2ym} and setting $\vec \phi= \vec u$ in \eqref{eq4.2xa}, we get
\[
\int_{\Omega} N_{ik}^\epsilon(\cdot,y) f^i = \int_ {\Omega_\epsilon(y)} \Phi_\epsilon u^k.
\]
We take the limit $\epsilon\to 0$ above to get
\[
u^k(y)=\int_\Omega N_{ik}(x,y) f^i(x)\,dx,
\]
which is equivalent to \eqref{eq2.9x}.
We have shown that $\vec N(x,y)$ satisfies the property iii) in Section~\ref{sec:nf}, and thus that $\vec N(x,y)$ is a unique Neumann function of the operator $L$ in $\Omega$.

Finally, let $\vec f\in L^q(\Omega)^m$ with $q>d/2$ and $\vec g\in L^2(\partial\Omega)^m$ satisfy the compatibility condition \eqref{eq2.5iw}, and let $\vec u$ be a unique weak solution in $\tilde{W}^{1,2}(\Omega)^m$ of the problem \eqref{eq3.4rm}; see Section~\ref{sec:nbp}.
Then $\vec u$ satisfies the identity \eqref{eq2.6ux}.
By setting $\vec v$ to be the $k$-th column of $\tilde{\vec N}{}^\epsilon(\cdot,x)$ in \eqref{eq2.6ux} and utilizing an integral identity for $\tilde{\vec N}{}^\epsilon(\cdot,x)$ similar to \eqref{eq4.2xa}, we get
\[
\int_\Omega \tilde{N}_{ik}^\epsilon(\cdot,x) f^i +
\int_{\partial\Omega} \tilde{N}_{ik}^\epsilon(\cdot,x) g^i = \int_ {\Omega_\epsilon(x)} \Phi_\epsilon u^k.
\]
We remark that the condition \eqref{IH} together with the assumption that $\vec f\in L^q(\Omega)^m$ with $q>d/2$ implies that $\vec u$ is H\"older continuous in $\Omega$; see e.g., \cite[Section~3.2]{HK07}.
Then by proceeding similarly as above and using \eqref{eq3.01mq}, we obtain
\[
u^k(x)=\int_\Omega N_{ki}(x,y) f^i(y)\,dy+\int_\Omega N_{ki}(x,y) g^i(y)\,d\sigma(y),
\]
which is the formula \eqref{eqM1.e}.
The proof is complete.
\hfill\qedsymbol

\subsection{Proof of Theorem~\ref{thm2}}			\label{sec4.2}
First, we shall assume that $0<\abs{x-y} \le 2/3$ and prove the bound \eqref{eq2.17dc}.
Let $0<R<1$ and $y\in \Omega$ be arbitrary, but fixed.
Assume that $\vec f \in C_c^\infty(\Omega)^m$ is supported in $\Omega_R(y)$ and let $\vec u$ be a unique weak solution in $\tilde{W}^{1,2}(\Omega)^m$ of the problem \eqref{eq2.10yq}.
Then we have the identities \eqref{eq4.2ym} and \eqref{eq4.3wx} as in the proof of Theorem~\ref{thm1}.
Also, we have the estimate \eqref{eq4.4ur}, and thus by \eqref{eq2.3cd} we get
\begin{equation}        \label{eq3.1a}
\norm{\vec u}_{L^{2d/(d-2)}(\Omega)} \le C\norm{D\vec u}_{L^2(\Omega)}\le C \norm{\vec f}_{L^{2d/(d+2)}(\Omega)}\le C  R^{(2+d)/2} \norm{\vec f}_{L^\infty(\Omega_R(y))},
\end{equation}
where $C=C(d, m, \lambda, M, \Omega)$.
Observe that $\vec g:=-(1/\abs{\partial\Omega})\int_\Omega \vec f$ has the bound
\[
\abs{\vec g} \le \frac{1}{\abs{\partial\Omega}} \int_\Omega \abs{\vec f} \le  C R^d \norm{\vec f}_{L^\infty(\Omega_R(y))},
\]
where $C=C(d,\Omega)$, and thus we have
\begin{equation}					\label{eq4.22hh}
R \norm{\vec  g}_{L^\infty(\Sigma_R(y))}  \le C R^{d+1} \norm{\vec f}_{L^\infty(\Omega_R(y))} \le C R^2 \norm{\vec f}_{L^\infty(\Omega_R(y))},
\end{equation}
where we used the assumption that $R <1$.
Then by \eqref{LB}, \eqref{eq3.1a}, \eqref{eq4.22hh}, and H\"older's inequality, we obtain
\begin{equation}    \label{eq3.2v}
\norm{\vec u}_{L^\infty(\Omega_{R/2}(y))} \le C R^2 \norm{\vec f}_{L^\infty(\Omega_R(y))},
\end{equation}
where $C=C(d,m,\lambda, M, \Omega, C_1)$.
Hence, by \eqref{eq4.3wx} and \eqref{eq3.2v}, we conclude that
\begin{equation}        \label{eq3.3y}
\Biggabs{\int_{\Omega_R(y)} \vec N^\epsilon(\cdot,y)^T \vec f \,} \le CR^2 \norm{\vec f}_{L^\infty(\Omega_R(y))},\quad \forall \vec f\in C_c^\infty(\Omega_R(y)),\;\; \forall \epsilon \in (0,R/2).
\end{equation}
Therefore, by duality, we conclude from \eqref{eq3.3y} that
\begin{equation}    \label{eq3.3w}
\norm{\vec N^\epsilon (\cdot,y)}_{L^1(\Omega_R(y))}\le C R^2,\quad \forall \epsilon \in (0,R/2),
\end{equation}
where $C=C(d,m,\lambda, M, \Omega, C_1)$.

Next, recall that the $\vec v=\vec v_{\epsilon,y,k}$ (i.e., $k$-th column of $\vec N^\epsilon(\cdot,y)$) is a unique weak solution in $\tilde{W}^{1,2}(\Omega)^m$ of the problem \eqref{eq4.0ik}.
Let $x\in\Omega$,  $r>0$, and $\epsilon>0$ be such that $B_\epsilon(y)\cap B_r(x) =\emptyset$.
Then, the condition \eqref{LB} implies that
\begin{equation}					\label{eq4.25dd}
\norm{\vec N^\epsilon(\cdot, y)}_{L^\infty(\Omega_{r/2}(x))} \le C_1 \left( r^{-d/2} \norm{\vec N^\epsilon(\cdot, y)}_{L^2(\Omega_r(x))}+ \abs{\partial\Omega}^{-1} r\right).
\end{equation}
By a standard iteration argument (see \cite[pp. 80--82]{Gi93}), we then obtain from \eqref{eq4.25dd} that
\begin{equation}    \label{eq2.8r}
\norm{\vec N^\epsilon(\cdot, y)}_{L^\infty(\Omega_{r/2}(x))}  \le C r^{-d} \norm{\vec N^\epsilon(\cdot, y)}_{L^1(\Omega_r(x))}+ C r,
\end{equation}
where $C=C(d, C_1,\abs{\partial\Omega})$.

Now, for any $x\in\Omega$ satisfying $0<\abs{x-y}\le 2/3$, take $R=3r=3\abs{x-y}/2$.
Then by \eqref{eq2.8r} and \eqref{eq3.3w}, we obtain for all $\epsilon \in (0,r)$ that
\begin{multline}					\label{eq4.28ej}
\abs{\vec N^\epsilon(x,y)} \le C r^{-d} \norm{\vec N^\epsilon(\cdot, y)}_{L^1(\Omega_r(x))} +Cr \le C r^{-d} \norm{\vec N^\epsilon(\cdot,y)}_{L^1(\Omega_{3r}(y))}+ C r\\
\le C R^{2-d} + C R \le C R^{2-d} \le C \abs{x-y}^{2-d},
\end{multline}
where $C=C(d,m,\lambda, M, \Omega, C_1)$ and we have again used the assumption that $R \le 1$.
Therefore, by using \eqref{eq4.17ht}, we may take the limit $\epsilon\to 0$ in the above inequality  and obtain \eqref{eq2.17dc} under an extra assumption that $\abs{x-y}\le 2/3$.
In the case when $\abs{x-y} >2/3$, we take $R=3r=1$ in \eqref{eq4.28ej} and get 
\[
\abs{\vec N^\epsilon(x,y)} \le C \le C\diam(\Omega)^{d-2} \abs{x-y}^{2-d} \le C \abs{x-y}^{2-d},
\]
where $C=C(d,m,\lambda, M, \Omega, C_1)$.
Again, by taking the limit $\epsilon\to 0$ in the above inequality, we obtain \eqref{eq2.17dc} even if $\abs{x-y} >2/3$.
We have thus shown that $\LB$ implies \eqref{eq2.17dc}.

To derive the estimates i) -- v) in the theorem, we need to repeat some steps in the proof of Theorem~\ref{thm1} with a little modification.
Let $\vec v_\epsilon$ be the $k$-th column of $\vec N^\epsilon(\cdot,y)$, where $k=1,\ldots, m$, $0<\epsilon<\min(d_y,r)/6$, and $0<r<\diam(\Omega)$.
Let $\eta$ be a smooth function on $\bR^d$ satisfying the conditions \eqref{eq4.19h}.
We set $\vec \phi=\eta^2\vec v_\epsilon$ in \eqref{eq4.1tv} to get
\begin{equation}					\label{eq4.34aa}
\int_{\Omega}\eta^2 A^{\alpha\beta}_{ij} D_\beta v_\epsilon^j D_\alpha v_\epsilon^i+\int_{\Omega}2 \eta A^{\alpha\beta}_{ij} D_\beta v_\epsilon^j D_\alpha \eta v_\epsilon^i+ \frac{1}{\abs{\partial\Omega}}\int_{\partial\Omega} (\eta^2-1) v_\epsilon^k = 0,
\end{equation}
where we used the fact $\int_{\partial\Omega} \vec v_\epsilon =0$ and $\eta^2 \Phi_\epsilon \equiv 0$.
We then use \eqref{eqP-02}, \eqref{eqP-03}, Cauchy's inequality to get
\[
\int_\Omega \eta^2 \abs{D \vec N^\epsilon(\cdot,y)}^2 \le C \left( \int_{\Omega} \abs{D \eta}^2 \abs{\vec N^\epsilon(\cdot,y)}^2+\frac{1}{\abs{\partial\Omega}}\int_{\partial\Omega} (1-\eta^2) \abs{\vec N^\epsilon(\cdot,y)} \right),
\]
where $C=C(\lambda, M)$.
By using the conditions in \eqref{eq4.19h} and the pointwise bound for $\vec N^\epsilon(x,y)$ obtained above, we get 
\begin{align*}
\int_{\Omega\setminus B_r(y)} \abs{D \vec N^\epsilon(\cdot,y)}^2 & \le C \left( r^{-2} \int_{B_r(y)\setminus B_{r/2}(y)} \abs{x-y}^{4-2d}\,dx+\frac{1}{\abs{\partial\Omega}}\int_{\Sigma_r(y)}  \abs{x-y}^{2-d}\,d\sigma(x) \right)\\
&\le C r^{2-d},\quad\text{where }\;C=C(d,m,\lambda, M, \Omega, C_1).
\end{align*}
Therefore, by taking the limit $\epsilon\to 0$, we get
\[
\norm{D\vec N(\cdot,y)}_{L^2(\Omega\setminus B_r(y))}\le C r^{(2-d)/2},\quad0<\forall r<\diam(\Omega).
\]
Observe that the pointwise bound \eqref{eq2.17dc} together with the above estimate yields
\begin{equation}							\label{eq4.26gg}
\norm{\vec N(\cdot,y)}_{L^{2d/(d-2)}(\Omega\setminus B_r(y))}+\norm{D\vec N(\cdot,y)}_{L^2(\Omega\setminus B_r(y))}\le C r^{(2-d)/2},\quad0<\forall r<\diam(\Omega),
\end{equation}
where $C=C(d,m,\lambda, M, \Omega, C_1)$.
By following literally the same step used in deriving \eqref{eq4.12eq} -- \eqref{eq4.15tw} from \eqref{eq4.11bs}, and using the fact that $\abs{\Omega}<\infty$, we obtain the estimates i) -- v) from \eqref{eq4.26gg}.
The proof is complete.
\hfill\qedsymbol

\subsection{Proof of Theorem~\ref{thm3}}
By the symmetry, it is enough to prove \eqref{LB} for weak solutions of  the problem
\[
\left\{
\begin{aligned}
\Lt\vec u&=\vec f \quad\text{in }\;\Omega,\\
{}^t\!\vec A D\vec u \cdot \vec n&= \vec g \quad \text{on }\;\partial\Omega,
\end{aligned}
\right.
\]
where $\vec f \in C_c^\infty(\Omega)^m$ and $\vec g \in C^\infty(\partial\Omega)^m$ are such that $\int_\Omega \vec f + \int_{\partial\Omega} \vec g =0$.

Let $\vec u$ be a unique weak solution in $\tilde{W}^{1,2}(\Omega)^m$ of the above problem.
We then have the identity (c.f. \eqref{eq2.6ux})
\begin{equation}                \label{eq4.23h}
\int_\Omega A^{\alpha\beta}_{ij} D_\beta w^j D_\alpha u^i = \int_{\partial\Omega} g^i w^i + \int_\Omega f^i w^i,\quad \forall \vec w \in W^{1,2}(\Omega)^m.
\end{equation}
Let  $\zeta$ be a smooth function on $\bR^d$ satisfying
\begin{equation}                \label{eq4zeta}
0\le \zeta\le 1,\quad \supp \zeta \subset B_{R/2}(x),\quad \zeta\equiv 1\,\text{ on }\, B_{3R/8}(x),\quad\text{and}\quad \abs{D\zeta} \le 16/R.
\end{equation}
We set $\vec w= \zeta \vec v_\epsilon$  in \eqref{eq4.23h}, where $\vec v_\epsilon=\vec v_{\epsilon,y,k}$ is the $k$-th column of $\vec N^\epsilon(\cdot,y)$, to get
\begin{equation}					\label{eq4.28sz}
\int_\Omega \zeta A^{\alpha\beta}_{ij} D_\beta v_\epsilon^j D_\alpha u^i =- \int_\Omega A^{\alpha\beta}_{ij} D_\beta \zeta  v_\epsilon^j D_\alpha u^i + \int_{\partial\Omega} \zeta g^i v_\epsilon^i+\int_\Omega \zeta f^i v_\epsilon^i.
\end{equation}
On the other hand, by setting $\vec \phi= \zeta \vec u$ in \eqref{eq4.1tv}, we get 
\begin{equation}				\label{eq4.29bd}
\int_ {\Omega_\epsilon(y)} \Phi_\epsilon \zeta u^k=\int_\Omega \zeta A^{\alpha\beta}_{ij} D_\beta v_\epsilon^j D_\alpha u^i + \int_\Omega A^{\alpha\beta}_{ij} D_\beta v_\epsilon^j D_\alpha \zeta u^i + \frac{1}{\abs{\partial\Omega}}\int_{\partial\Omega} \zeta u^k.
\end{equation}
Therefore, by combining \eqref{eq4.28sz} and \eqref{eq4.29bd}, we obtain
\begin{multline*}
\int_ {\Omega_\epsilon(y)} \Phi_\epsilon \zeta u^k= \int_\Omega A^{\alpha\beta}_{ij} D_\beta v_\epsilon^j D_\alpha \zeta u^i + \frac{1}{\abs{\partial\Omega}}\int_{\partial\Omega} \zeta u^k- \int_\Omega A^{\alpha\beta}_{ij} D_\beta \zeta  v_\epsilon^j D_\alpha u^i \\
- \int_{\partial\Omega} \zeta g^i v_\epsilon^i+\int_\Omega \zeta f^i v_\epsilon^i.
\end{multline*}

Now, assume that $y\in \Omega_{R/4}(x)$.
Notice from \eqref{eq4zeta} that $\dist(y, \supp D\zeta)> R/8$. 
Then by taking $\epsilon \to 0$ in the above identity, we get
\begin{multline}                   \label{eq4.32uq}
u^k(y)= \int_\Omega A^{\alpha\beta}_{ij} D_\beta N_{jk}(\cdot,y) D_\alpha \zeta u^i + \frac{1}{\abs{\partial\Omega}}\int_{\partial\Omega} \zeta u^k- \int_\Omega A^{\alpha\beta}_{ij} D_\beta \zeta  N_{jk}(\cdot,y) D_\alpha u^i \\
- \int_{\partial\Omega} \zeta g^i N_{ik}(\cdot,y)+\int_\Omega \zeta f^i N_{jk}(\cdot,y)=: I_1 + I_2 + I_3+I_4+I_5.
\end{multline}

On the other hand, observe that $\eta^2 \vec v$, where $\vec v$ is the $k$-th column of $\vec N(\cdot,y)$ and $\eta$ satisfies the properties in \eqref{eq4.19h}, belongs to $W^{1,2}(\Omega)^m$. 
Then by approximation we may take $\phi=\eta^2 \vec v$ in \eqref{eq2.6r} to get
\[
\int_{\Omega}\eta^2 A^{\alpha\beta}_{ij} D_\beta v^j D_\alpha v^i+\int_{\Omega} 2 \eta A^{\alpha\beta}_{ij} D_\beta v^j D_\alpha \eta v^i+ \frac{1}{\abs{\partial\Omega}}\int_{\partial\Omega} (\eta^2-1) v^k = 0,
\]
which corresponds to \eqref{eq4.34aa} in the proof of Theorem~\ref{thm2}.
Following exactly the same steps as in the proof of Theorem~\ref{thm2}, we then obtain the estimate \eqref{eq4.26gg}.
Also, from \eqref{eq4.23h} and the trace theorem, we derive Caccioppoli's inequality
\begin{equation}					\label{eq4.88vv}
\norm{D\vec u}_{L^2(\Omega_{R/2}(x))} \le C R^{-1}\norm{\vec u}_{L^2(\Omega_R(x))}+CR^{d/2}(1+R)\norm{\vec g}_{L^\infty(\Sigma_R(x))}+CR^{d/2+1}\norm{\vec f}_{L^\infty(\Omega_R(x))},
\end{equation}
where $C=C(d,m,\lambda, M,\Omega)$; see Appendix for the proof.

Denote $A_R(y)=\Omega_{3R/4}(y)\setminus B_{R/8}(y)$.
By using H\"older's inequality, \eqref{eq4.26gg}, and \eqref{eq2.17dd}, we estimate
\begin{align*}
\abs{I_1} &\le C R^{-1}\norm{D\vec N(\cdot,y)}_{L^2(A_R(y))}\, \norm{\vec u}_{L^2(\Omega_{R/2}(x))} \le C R^{-d/2}\norm{\vec u}_{L^2(\Omega_R(x))},\\
\abs{I_4} &\le C \norm{\vec N(\cdot,y)}_{L^1(\Sigma_{3R/4}(y))}\, \norm{\vec g}_{L^\infty(\Sigma_{R/2}(x))} \le CR \norm{\vec g}_{L^\infty(\Sigma_R(x))},\\
\abs{I_5} &\le C \norm{\vec N(\cdot,y)}_{L^1(\Omega_{3R/4}(y))}\, \norm{\vec f}_{L^\infty(\Omega_{R/2}(x))}\le  CR^2 \norm{\vec f}_{L^\infty(\Omega_R(x))}.
\end{align*}
Similarly, by H\"older's inequality, the trace theorem, and \eqref{eq4.88vv}, we estimate
\begin{multline*}
\abs{I_2} \le \abs{\partial\Omega}^{-1/2} \norm{\zeta\vec u}_{L^2(\partial\Omega)}\le C(1+R^{-1}) \norm{\vec u}_{L^2(\Omega_R(x))}+ C\norm{D\vec u}_{L^2(\Omega_{R/2}(x))} \\
\le C(1+R^{-1})\,\norm{\vec u}_{L^2(\Omega_R(x))}+C(R^{d/2}+R^{d/2+1})\norm{\vec g}_{L^\infty(\Sigma_R(x))}+CR^{d/2+1}\norm{\vec f}_{L^\infty(\Omega_R(x))}.
\end{multline*}
Also, by H\"older's inequality, \eqref{eq4.26gg}, and \eqref{eq4.88vv}, we get
\begin{multline*}
\abs{I_3} \le C R^{-1} \norm{\vec N(\cdot,y)}_{L^2(A_R(y))}\, \norm{D\vec u}_{L^2(\Omega_{R/2}(x))}\le C R^{1-d/2}\norm{D\vec u}_{L^2(\Omega_{R/2}(x))}\\
\le CR^{-d/2}\norm{\vec u}_{L^2(\Omega_R(x))} + C R(1+R)\norm{\vec g}_{L^\infty(\Sigma_R(x))}+ R^2 \norm{\vec f}_{L^\infty(\Omega_R(x))}.
\end{multline*}
Combining together, we get from \eqref{eq4.32uq} that
\begin{multline*}
\norm{\vec u}_{L^\infty(\Omega_{R/4})} \le C R^{-d/2}(1+R^{d/2}+R^{d/2-1})\norm{\vec u}_{L^2(\Omega_R)}+CR(1+R^{d/2-1}+R^{d/2}+R)\norm{\vec g}_{L^\infty(\Sigma_R)}\\
+CR^2(1+R^{d/2-1})  \norm{\vec f}_{L^\infty(\Omega_R(x))},\quad\text{where }\;C=C(d,m,\lambda, M,\Omega).
\end{multline*}
By a standard covering argument and the fact that $R<\diam(\Omega)<\infty$, we obtain \eqref{LB} from the above inequality.
The proof is complete.
\hfill\qedsymbol

\mysection{Neumann functions in Lipschitz graph domain}				\label{lgd}
This separate section is devoted to the study of Neumann functions in an unbounded domain above a Lipschitz graph.
\subsection{Main results} 
Since Lipschitz graph domains are necessarily unbounded domains, it is more practical to replace the condition \eqref{IH} in Section~\ref{main} by the following condition \eqref{IHp}.
In the case when the domain is bounded, it is equivalent to the condition \eqref{IH}, but it is weaker if the domain is unbounded.
By the well known De Giorgi-Moser-Nash theorem, we have the condition \eqref{IHp} with $R_c=\infty$ in the scalar case and thus, it reduces to the condition \eqref{IH}.

\begin{CIHp}
There exist $\mu_0\in (0,1]$, $R_c\in (0,\infty]$, and $C_0>0$ such that for all $x\in\Omega$ and $R \in(0, d_x')$, where $d_x':=\min(d_x,R_c)$, the following holds:
If $\vec u\in W^{1,2}(B_R(x))$ is a weak solution of either $L\vec u=0$ or $\Lt \vec u=0$ in $B_R=B_R(x)$,  then $\vec u$ is H\"older continuous in $B_R$ with the following estimate:
\begin{equation}
\tag{IH$'$}\label{IHp}
[\vec u]_{C^{\mu_0}(B_{R/2})} \le C_0 R^{-\mu_0}\left(\fint_{B_R} \abs{\vec u}^2\right)^{1/2}.
\end{equation}
\end{CIHp}

\begin{theorem}		\label{thm4}
Let $\Omega$ be a Lipschitz graph domain in $\bR^d$ ($d\ge 3$).
Assume the condition \eqref{IHp}.
Then there exist Neumann functions $\vec N(x,y)$ of $L$ and $\tilde{\vec N}(x,y)$ of $\Lt$ in $\Omega$ satisfying the identity  \eqref{eq3.01mq}.
Furthermore, the estimates i) -- vii) in Theorem~\ref{thm1} are valid for $\vec N(\cdot,y)$ and $\tilde{\vec N}(\cdot,y)$ for all $y\in\Omega$ provided $d_x$ is replaced by $d_x'=\min(d_x, R_c)$.
\end{theorem}

We also replace the condition \eqref{LB} in Section~\ref{main} by the following condition \eqref{LBp}.
In the scalar case, it is well known that the condition \eqref{LBp} holds in Lipschitz graph domains.

\begin{CLBp}
There exists a constant $C_1>0$ such that the following holds:
For any $\vec f \in C_c^\infty(\Omega)^m$, let $\vec u \in  Y^{1,2}(\Omega)^m$ be a unique weak solution of the problem
\[
\begin{aligned}
\left\{
\begin{aligned}
L\vec u&=\vec f \quad\text{in }\;\Omega\\
\vec A D\vec u \cdot \vec n&= 0 \quad \text{on }\;\partial\Omega
\end{aligned}
\right. 
\qquad\;\text{ or }\;\qquad
\left\{
\begin{aligned}
\Lt\vec u&=\vec f \quad\text{in }\;\Omega\\
{}^t\!\vec A D\vec u \cdot \vec n&= 0\quad \text{on }\;\;\partial\Omega.
\end{aligned}
\right.
\end{aligned}
\]
Then for all $x\in\Omega$ and $R>0$, we have
\[			\tag{LB$'$}\label{LBp}
\norm{\vec u}_{L^\infty(\Omega_{R/2}(x))} \le C_1 \left(R^{-d/2} \norm{\vec u}_{L^2(\Omega_R(x))}+ R^2 \norm{\vec f}_{L^\infty(\Omega_R(x))} \right).
\]
\end{CLBp}

\begin{theorem}							\label{thm5}
Let $\Omega$ be a Lipschitz graph domain in $\bR^d$ ($d\ge 3$) with Lipschitz constant $K$ and assume the condition \eqref{IHp}.
If the condition \eqref{LBp} is also satisfied, then conclusions of Theorem~\ref{thm2} hold with $C=C(d,m,\lambda,M,K,C_1)$.
Conversely, suppose there exists a constant $C_2$ such that \eqref{eq2.17dd} holds.
Then the condition \eqref{LBp} is satisfied in $\Omega$ with $C_1=C_1(d,m,\lambda,M,K,C_2)$.
\end{theorem}

\subsection{Proof of Theorem~\ref{thm4}}
The proof is a slight modification of that of Theorem~\ref{thm1}.
Let $y\in\Omega$ be fixed but arbitrary.
For $\epsilon>0$ and $k=1,\ldots,m$, let $\vec v=\vec v_{\epsilon, y, k}$ be a unique weak solution in $Y^{1,2}(\Omega)^m$ of the problem
\begin{equation}					\label{eq5.0ik}
\left\{
\begin{aligned}
L \vec v &= \Phi_\epsilon \vec e_k\;\text{ in }\;\Omega,\\
\vec A D\vec v \cdot \vec n&=0 \;\text{ on }\; \partial\Omega,
\end{aligned}
\right.
\end{equation}
and define $\vec N^\epsilon(\cdot,y)$ by \eqref{eq4.1uu}.
Then $\vec N^\epsilon(\cdot,y)$ satisfies the identity
\begin{equation}				\label{eq6.1tv}
\int_{\Omega}A^{\alpha\beta}_{ij} D_\beta N_{jk}^\epsilon(\cdot,y) D_\alpha \phi^i = \int_ {\Omega_\epsilon(y)} \Phi_\epsilon \phi^k ,\quad
\forall \vec \phi \in Y^{1,2}(\Omega)^m.
\end{equation}
By the same argument as in the proof of Theorem~\ref{thm1}, we then obtain \eqref{eqG-02}.
Let $R\in (0,d_y')$ be arbitrary, but fixed.
Assume that $\vec f \in C_c^\infty(\Omega)^m$ is supported in $B_R=B_R(y) \subset \Omega$ and let $\vec u$ be a unique weak solution in $Y^{1,2}(\Omega)^m$ of the problem \eqref{eq5.10yq}.
Then, we get the identity \eqref{eq4.3wx} and also the estimate \eqref{eq4.4ur}.
By literally the same steps in the proof of Theorem~\ref{thm1}, we get  \eqref{eq4.5ac} and \eqref{eqG-20} -- \eqref{eq4.17ht} with $d_y$ replaced by $d_y'$.
Therefore, by the same reasoning as in the proof of Theorem~\ref{thm1}, we find that $\vec N(\cdot,y)$ satisfies the properties i) and ii) in Section~\ref{sec:nf2}, and also the estimates i) -- vii) in Theorem~\ref{thm1} with $d_y$ replaced by $d_y'$; see \cite[Section~4.1]{HK07}.
Let $\vec u$ be a unique solution in $Y^{1,2}(\Omega)^m$ of the problem \eqref{eq5.10yq} with $\vec f\in C_c^\infty(\Omega)^m$.
Then as in the proof of Theorem~\ref{thm1} again, we get \eqref{eq2.9x}.
Therefore, $\vec N(x,y)$ satisfies the property iii) in Section~\ref{sec:nf2}, and thus that $\vec N(x,y)$ is a unique Neumann function of the operator $L$ in $\Omega$.
The proof is complete.
\hfill\qedsymbol

\subsection{Proof of Theorem~\ref{thm5}}
We follow the proofs of Theorem~\ref{thm2} and \ref{thm3} with a few adjustment.
Let $y\in \Omega$ and $R>0$ be arbitrary, but fixed.
Assume that $\vec f \in C_c^\infty(\Omega)^m$ is supported in $\Omega_R(y)$ and let $\vec u$ be a unique weak solution in $Y^{1,2}(\Omega)^m$ of the problem
\[
\left\{
\begin{aligned}
\Lt \vec u &=\vec f\;\text{ in }\;\Omega,\\
{}^t\! \vec A D\vec u \cdot \vec n&= 0\;\text{ on }\;\partial\Omega.
\end{aligned}
\right.
\]
Then we have  \eqref{eq4.3wx} and \eqref{eq4.4ur} as in the proof of Theorem~\ref{thm2}, and thus by \eqref{eq5.3nh} we get \eqref{eq3.1a}.
By \eqref{LBp}, \eqref{eq3.1a}, and H\"older's inequality, we obtain \eqref{eq3.2v}.
Then by following the same steps as in the proof of Theorem~\ref{thm2} we get \eqref{eq3.3w} with $C=C(d,m,\lambda, M, K, C_1)$.

Let $x\in\Omega$,  $r>0$, and $\epsilon>0$ be such that $B_\epsilon(y)\cap B_r(x) =\emptyset$.
Since  the $k$-th column of $\vec N^\epsilon(\cdot,y)$ is a unique weak solution in $Y^{1,2}(\Omega)^m$ of the problem \eqref{eq5.0ik}, the condition \eqref{LBp} implies that
\[
\norm{\vec N^\epsilon(\cdot, y)}_{L^\infty(\Omega_{r/2}(x))} \le C_1 r^{-d/2} \norm{\vec N^\epsilon(\cdot, y)}_{L^2(\Omega_r(x))}.
\]
Then by following literally the same steps as in the proof of Theorem~\ref{thm2}, we obtain the desired pointwise bound \eqref{eq2.17dc}.

Next, let $\vec v_\epsilon$ be the $k$-th column of $\vec N^\epsilon(\cdot,y)$, where $k=1,\ldots, m$, $0<\epsilon<\min(d_y',r)/6$, and $r>0$.
Let $\eta$ be a smooth function on $\bR^d$ satisfying the conditions in \eqref{eq4.19h}.
We set $\vec \phi=\eta^2\vec v_\epsilon$ in \eqref{eq6.1tv} to get the Caccioppoli's  inequality
\begin{equation}					\label{eq5.10ng}
\int_\Omega \eta^2 \abs{D \vec N^\epsilon(\cdot,y)}^2 \le C \int_{\Omega} \abs{D \eta}^2 \abs{\vec N^\epsilon(\cdot,y)}^2.
\end{equation}
By using the conditions in \eqref{eq4.19h} for $\eta$ and the pointwise bound \eqref{eq2.17dc}, we get 
\[
\int_{\Omega\setminus B_r(y)} \abs{D \vec N^\epsilon(\cdot,y)}^2 \le C r^{-2} \int_{B_r(y)\setminus B_{r/2}(y)} \abs{x-y}^{4-2d}\,dx \le C r^{2-d}.
\]
Therefore, by taking the limit $\epsilon\to 0$, we get
\[
\norm{D\vec N(\cdot,y)}_{L^2(\Omega\setminus B_r(y))}\le C r^{(2-d)/2},\quad \forall r>0.
\]
The pointwise bound \eqref{eq2.17dc} together with the above estimate yields
\begin{equation}							\label{eq6.26gg}
\norm{\vec N(\cdot,y)}_{L^{2d/(d-2)}(\Omega\setminus B_r(y))}+\norm{D\vec N(\cdot,y)}_{L^2(\Omega\setminus B_r(y))}\le C r^{(2-d)/2},\quad \forall r>0,
\end{equation}
where $C=C(d,m,\lambda, M, K, C_1)$.
By following literally the same step used in deriving \eqref{eq4.12eq} -- \eqref{eq4.15tw} from \eqref{eq4.11bs} we obtain from \eqref{eq6.26gg} the estimates i) -- v) in Theorem~\ref{thm2} with constants $C=C(d,m,\lambda, M, K, C_1)$.

It remains to show that the pointwise bound \eqref{eq2.17dc} implies the condition \eqref{LBp}.
By the symmetry, it is enough to prove \eqref{LBp} for a weak solution $\vec u\in Y^{1,2}(\Omega)^m$ of  the problem
\[
\left\{
\begin{aligned}
\Lt\vec u&=\vec f \quad\text{in }\;\Omega,\\
{}^t\!\vec A D\vec u \cdot \vec n&=0 \quad \text{on }\;\partial\Omega.
\end{aligned}
\right.
\]
Let $\vec u$ be a unique weak solution in $Y^{1,2}(\Omega)^m$ of the above problem, where $\vec f\in C_c^\infty(\Omega)^m$, so that we have the identity 
\begin{equation}                \label{eq6.23h}
\int_\Omega A^{\alpha\beta}_{ij} D_\beta w^j D_\alpha u^i =  \int_\Omega f^i w^i,\quad \forall \vec w \in Y^{1,2}(\Omega)^m.
\end{equation}
We set $\vec w= \zeta \vec v_\epsilon$ in \eqref{eq6.23h}, where $\zeta$ is as in \eqref{eq4zeta} and $\vec v_\epsilon$ is the weak solution in $Y^{1,2}(\Omega)^m$ of the problem \eqref{eq5.0ik} (i.e., $\vec v_\epsilon$ is the $k$-th column of $\vec N^\epsilon(\cdot,y)$), to get
\begin{equation}					\label{eq6.28sz}
\int_\Omega \zeta A^{\alpha\beta}_{ij} D_\beta v_\epsilon^j D_\alpha u^i =- \int_\Omega A^{\alpha\beta}_{ij} D_\beta \zeta  v_\epsilon^j D_\alpha u^i +\int_\Omega \zeta f^i v_\epsilon^i.
\end{equation}
On the other hand, by setting $\vec \phi= \zeta \vec u$ in \eqref{eq6.1tv}, we get 
\begin{equation}				\label{eq6.29bd}
\int_ {\Omega_\epsilon(y)} \Phi_\epsilon \zeta u^k=\int_\Omega \zeta A^{\alpha\beta}_{ij} D_\beta v_\epsilon^j D_\alpha u^i + \int_\Omega A^{\alpha\beta}_{ij} D_\beta v_\epsilon^j D_\alpha \zeta u^i.
\end{equation}
Therefore, by combining \eqref{eq6.28sz} and \eqref{eq6.29bd}, we obtain
\[
\int_ {\Omega_\epsilon(y)} \Phi_\epsilon \zeta u^k= \int_\Omega A^{\alpha\beta}_{ij} D_\beta v_\epsilon^j D_\alpha \zeta u^i - \int_\Omega A^{\alpha\beta}_{ij} D_\beta \zeta  v_\epsilon^j D_\alpha u^i+\int_\Omega \zeta f^i v_\epsilon^i.
\]
Assume $y\in \Omega_{R/4}(x)$ and take $\epsilon \to 0$ in the above identity to get (c.f. \eqref{eq4.32uq}) 
\begin{multline*}
u^k(y)= \int_\Omega A^{\alpha\beta}_{ij} D_\beta N_{jk}(\cdot,y) D_\alpha \zeta u^i - \int_\Omega A^{\alpha\beta}_{ij} D_\beta \zeta  N_{jk}(\cdot,y) D_\alpha u^i \\
+\int_\Omega \zeta f^i N_{jk}(\cdot,y)=: I_1 + I_2 + I_3.
\end{multline*}

On the other hand, by using the fact that $C_c^\infty(\overline\Omega)$ is dense in $Y^{1,2}(\Omega)$ (see the proof of Lemma~\ref{lem6.8ap} in Appendix), we may set $\phi=\eta^2 \vec v$ in \eqref{eq5.6r}, where $\vec v$ is the $k$-th column of $\vec N(\cdot,y)$ and $\eta$ satisfies the properties in \eqref{eq4.19h}, to get the following  inequality (c.f. \eqref{eq5.10ng}):
\[
\int_\Omega \eta^2 \abs{D \vec N(\cdot,y)}^2 \le C \int_{\Omega} \abs{D \eta}^2 \abs{\vec N(\cdot,y)}^2.
\]
Then by proceeding as before, we again obtain the estimate \eqref{eq6.26gg}.
With aid of \eqref{eq5.3nh}, we also derive the following Caccioppoli's inequality from \eqref{eq6.23h}:
\begin{equation}					\label{eq6.88vv}
\norm{D \vec u}_{L^2(\Omega_{R/2}(x))} \leq CR^{-1} \norm{\vec u}_{L^2(\Omega_R(x))}+C\norm{\vec f}_{L^{2d/(d+2)}(\Omega_R(x))}.
\end{equation}
Now, denote $A_R(y)=\Omega_{3R/4}(y)\setminus B_{R/8}(y)$.
Recall that $\zeta$ satisfies the properties in \eqref{eq4zeta}.
Then by H\"older's inequality and \eqref{eq6.26gg}, we estimate
\[
\abs{I_1} \le C R^{-1}\norm{D\vec N(\cdot,y)}_{L^2(A_R(y))}\, \norm{\vec u}_{L^2(\Omega_{R/2}(x))} \le C R^{-d/2}\norm{\vec u}_{L^2(\Omega_{R}(x))}.
\]
Similarly, by H\"older's inequality, \eqref{eq6.26gg}, and \eqref{eq6.88vv}, we estimate
\[
\abs{I_2} \le C R^{-1} \norm{\vec N(\cdot,y)}_{L^2(A_R(y))}\, \norm{D\vec u}_{L^2(\Omega_{R/2}(x))} \le C R^{-d/2}\norm{\vec u}_{L^2(\Omega_R(x))}+R^2\norm{\vec f}_{L^\infty(\Omega_{R}(x))}.
\]
Finally, by H\"older's inequality and \eqref{eq2.17dd}, we estimate
\[
\abs{I_3} \le C \norm{\vec N(\cdot,y)}_{L^1(\Omega_{3R/4}(y))}\, \norm{\vec f}_{L^\infty(\Omega_{R/2}(x))}\le C R^2 \norm{\vec f}_{L^\infty(\Omega_{R}(x))}.
\]
Combining the above estimates and using a standard covering argument, we obtain \eqref{LBp}.
The proof is complete.
\hfill\qedsymbol

\begin{acknowledgment}
We thank Russell Brown and Hongjie Dong for valuable discussion and comments.
We also thank the referee for helpful suggestions.
This work was supported by Basic Science Research Program through the National Research Foundation of Korea(NRF) funded by the Ministry of Education, Science and Technology (2010-0008224).
Seick Kim is supported by WCU(World Class University) program through the National Research Foundation of Korea(NRF) funded by the Ministry of Education, Science and Technology (R31-10049) and also by TJ Park Junior Faculty Fellowship.
\end{acknowledgment}

\section{Appendix}
\begin{lemma}
Let $\Omega$ be a bounded $C^1$ domain in $\bR^d$.
Suppose that the coefficients $A^{\alpha\beta}_{ij}$ of the system \eqref{eq0.0} belong to the VMO class and satisfy the conditions \eqref{eqP-02} and \eqref{eqP-03}.
Then the condition \eqref{LB} is satisfied.
\end{lemma}
\begin{proof}
Assume that  $\vec f \in C_c^\infty(\Omega)^m$ and $\vec g \in C^\infty(\partial\Omega)^m$ satisfy the compatibility condition $\int_\Omega \vec f + \int_{\partial\Omega} \vec g =0$ and let $\vec u \in \tilde{W}^{1,2}(\Omega)^m$ be a unique weak solution of the problem
\[
\left\{
\begin{aligned}
L\vec u&=\vec f \;\text{ in }\;\Omega,\\
\vec A D\vec u \cdot \vec n&=  \vec g \; \text{ on }\;\partial\Omega.
\end{aligned}
\right. 
\]
Let $\vec v=\zeta \vec u$, where $\zeta:\bR^d \to\bR$ is a smooth function to be chosen later, and observe that $\vec v$ is a weak solution in $W^{1,2}(\Omega)^m$ of the problem
\[
\left\{
\begin{aligned}
-D_\alpha(A^{\alpha\beta}_{ij} D_\beta v^j)&= \zeta f^i-\Psi^i-D_\alpha F^i_\alpha \;\text{ in }\;\Omega,\\
A^{\alpha\beta}_{ij} D_\beta v^j n_\alpha &= \zeta g^i+F^i_\alpha n_\alpha \; \text{ on }\;\partial\Omega,
\end{aligned}
\right. 
\]
where we used the notation
\[
\Psi^i=A^{\alpha\beta}_{ij} D_\alpha \eta D_\beta u^j ,\quad F^i_\alpha=A^{\alpha\beta}_{ij} D_\beta \zeta u^j.
\]
For $i=1,\ldots, m$, let $w^i$ be a solution of the Neumann problem
\[
\left\{
\begin{aligned}
-\Delta w^i&= \zeta f^i-\Psi^i \quad\text{in }\;\Omega,\\
\partial w^i/\partial n &=  \zeta g^i \quad \text{on }\;\partial\Omega.
\end{aligned}
\right.
\]
Then, by \cite[Corollary~9.3]{FMM} together with the embedding theorems of Sobolev and Besov spaces (see e.g., \cite{BL}), we have the following estimate for $D w^i$ provided $p>d/(d-1)$:
\[
\norm{D \vec w}_{L^p(\Omega)} \leq C \left(\norm{\zeta \vec f}_{L^{pd/(p+d)}(\Omega)}+\norm{\vec \Psi}_{L^{pd/(p+d)}(\Omega)}+\norm{\zeta \vec g}_{L^{p(d-1)/d}(\partial\Omega)}\right).
\]
Notice that if we set $h^i_\alpha=D_\alpha w^i+F^i_\alpha$, then $\vec v$ becomes a weak solution of the problem
\[
\left\{
\begin{aligned}
D_\alpha(A^{\alpha\beta}_{ij} D_\beta v^j)&= D_\alpha h^i_\alpha \;\text{ in }\;\Omega\\
(A^{\alpha\beta}_{ij} D_\beta v^j -h^i_\alpha) n_\alpha &=  0 \; \text{ on }\;\partial\Omega.
\end{aligned}
\right.
\]
We then apply \cite[Theorem~1]{BCKW} to conclude that $\vec v \in W^{1,p}(\Omega)^m$ with the estimate
\begin{equation}					\label{eq6.4ao}
\norm{D \vec v}_{L^p(\Omega)} \le C \left(\norm{\zeta \vec f}_{L^{pd/(p+d)}(\Omega)}+\norm{\vec \Psi}_{L^{pd/(p+d)}(\Omega)}+\norm{\zeta \vec g}_{L^{p(d-1)/d}(\partial\Omega)}+\norm{\vec F}_{L^p(\Omega)}\right).
\end{equation}
By choosing $\zeta\equiv 1$, we find that $\vec u \in W^{1,p}(\Omega)^m$ and 
\[
\norm{D \vec u}_{L^p(\Omega)} \le C \left(\norm{\vec f}_{L^{pd/(p+d)}(\Omega)}+\norm{\vec g}_{L^{p(d-1)/d}(\partial\Omega)}\right),
\]
and thus, via Morrey's imbedding theorem, we find that $\vec u \in C^\mu(\overline \Omega)$ for any $\mu\in (0,1)$, which particularly implies that $\vec u$ is globally bounded in $\Omega$.

To obtain \eqref{LB}, we employ the standard localization method as follows.
Let $x\in\Omega$ and $0<R<\diam(\Omega)$ be arbitrary but fixed.
For any $y\in \Omega\cap B_R(x)$ and $0<\rho<r \le R$, we choose the function $\zeta$ such that
\[
0\le \zeta \le 1,\quad \supp \zeta \subset B_r(y),\quad \zeta\equiv 1\,\text{ on }\, B_\rho(y),\quad\text{and}\quad\abs{D \zeta} \le 2/(r-\rho).
\]
Recall that we use the notation
\[
\Omega_r=\Omega_r(y)=\Omega\cap B_r(y),\quad \Sigma_r=\Sigma_r(y)=\partial\Omega\cap B_r(y).
\]
Then by using the assumptions on $\zeta$, we estimate terms in \eqref{eq6.4ao} as follows.
\begin{align*}
\norm{\zeta \vec f}_{L^{pd/(p+d)}(\Omega)} &\le C r^{1+d/p}\norm{\vec f}_{L^\infty(\Omega_r)},\\
\norm{\vec \Psi}_{L^{pd/(p+d)}(\Omega)} &\le C(r-\rho)^{-1}\norm{D\vec u}_{L^{pd/(p+d)}(\Omega_r)}, \\
\norm{\zeta \vec g}_{L^{p(d-1)/d}(\Omega)} &\le C r^{d/p}\norm{\vec g}_{L^\infty(\Sigma_r)}, \\
\norm{\vec F}_{L^p(\Omega)} &\le C(r-\rho)^{-1}\norm{\vec u}_{L^p(\Omega_r)}.
\end{align*}
By using the inequality \eqref{eq6.4ao} and the above estimates, we get
\begin{multline}				\label{eq6.4he}
\norm{D \vec u}_{L^p(\Omega_\rho)} \le C r^{1+d/p}\norm{\vec f}_{L^\infty(\Omega_r)}+C r^{d/p} \norm{\vec g}_{L^\infty(\Sigma_r)}\\
+C(r-\rho)^{-1}\norm{\vec u}_{L^p(\Omega_r)}+C(r-\rho)^{-1} \norm{D\vec u}_{L^{pd/(p+d)}(\Omega_r)}.
\end{multline}
We fix $p>d$ and let $k$ be the smallest integer such that $k\ge d(1/2-1/p)$.
We set
\[
p_i=pd/(d+p i)\quad\text{and}\quad r_i=\rho+(r-\rho)i/k,\quad i=0,\ldots,k.
\]
Then we apply \eqref{eq6.4he} iteratively to get
\begin{multline*}
\norm{D \vec u}_{L^p(\Omega_\rho)}
\le \sum_{i=1}^k C^i\left(\frac{k}{r-\rho}\right)^{i-1}\left( r_i^{1+d/p_{i-1}} \norm{\vec f}_{L^\infty(\Omega_{r_i})}+ r_i^{d/p_{i-1}} \norm{\vec g}_{L^\infty(\Sigma_{r_i})}\right) \\
+\sum_{i=1}^k C^i\left(\frac{k}{r-\rho}\right)^i \norm{\vec u}_{L^{p_{i-1}}(\Omega_{r_i})}+C^k\left(\frac{k}{r-\rho}\right)^k\norm{D\vec u}_{L^{p_k}(\Omega_{r_k})}.
\end{multline*}
Notice that $1<p_k\le 2$.
By using H\"older's inequality we then obtain
\begin{multline*}
\rho^{-d(1/2-1/p)}\norm{D \vec u}_{L^2(\Omega_\rho)} \le C\left(\frac{r}{r-\rho}\right)^{k-1} \left(r^{1+d/p}\norm{\vec f}_{L^\infty(\Omega_r)}+r^{d/p} \norm{\vec g}_{L^\infty(\Sigma_r)}\right)\\
+C \left(\frac{r}{r-\rho}\right)^k r^{-1} \norm{\vec u}_{L^p(\Omega_r)}+ C\left(\frac{r}{r-\rho}\right)^k r^{d(1/p-1/2)} \norm{D\vec u}_{L^2(\Omega_r)}.
\end{multline*}
If we take $r=R/4$ and $\rho<r/2=R/4$ in the above, then for all $y\in \Omega_{R/4}(x)$, we get
\begin{multline}					\label{eq6.5tt}
\left(\rho^{-(d-2+2(1-d/p))}\int_{\Omega_\rho(y)}\abs{D\vec u}^2\right)^{1/2}
 \le  C R^{1+d/p}\norm{\vec f}_{L^\infty(\Omega_R(x))}+CR^{d/p} \norm{\vec g}_{L^\infty(\Sigma_R(x))}\\
+C R^{-1} \norm{\vec u}_{L^p(\Omega_R(x))}+CR^{d(1/p-1/2)} \norm{D\vec u}_{L^2(\Omega_{R/2}(x))}=:A(R).
\end{multline}
Hereafter in the proof, we shall denote $\Omega_R=\Omega_R(x)$.
Then by Morrey-Campanato's theorem (see \cite[Section~3.1]{Gi93}), for all $z, z' \in \Omega_{R/4}$, we have
\[
\abs{\vec u(z)-\vec u(z')} \le C R^{1-d/p} A(R),
\]
where $A(R)$ is as defined in \eqref{eq6.5tt}.
Therefore, for any $z\in \Omega_{R/4}$ we have
\[
\abs{\vec u(z)} \le \abs{\vec u(z')}+ \abs{\vec u(z)-\vec u(z')} \le \abs{\vec u(z')} + C R^{1-d/p} A(R),\quad \forall z'\in \Omega_{R/4}.
\]
By taking average over $z'\in \Omega_{R/4}$ in the above and using the definition of $A(R)$, we obtain
\begin{multline*}
\sup_{\Omega_{R/4}}\,\abs{\vec u} 
\le \fint_{\Omega_{R/4}} \abs{\vec u(z')}\,dz'+ C R^2\norm{\vec f}_{L^\infty(\Omega_R)}+CR \norm{\vec g}_{L^\infty(\Sigma_R)}\\
+C R^{-d/p} \norm{\vec u}_{L^p(\Omega_R)}+CR^{1-d/2} \norm{D\vec u}_{L^2(\Omega_{R/2})}.
\end{multline*}
Then by using H\"older's inequality, Caccioppoli's inequality (see Lemma~\ref{lem:a2} below), and the fact that $\Omega$ is bounded, we get
\[
\sup_{\Omega_{R/4}}\,\abs{\vec u} 
\le C R^2\norm{\vec f}_{L^\infty(\Omega_R)}+CR \norm{\vec g}_{L^\infty(\Sigma_R)}+C R^{-d/p} \norm{\vec u}_{L^p(\Omega_R)}+CR^{-d/2} \norm{\vec u}_{L^2(\Omega_R)}.
\]
By using a standard argument (see \cite[pp. 80--82]{Gi93}), we derive from the above inequality 
\[
\sup_{\Omega_{R/2}}\, \abs{\vec u} \le
CR^2\norm{\vec f}_{L^\infty(\Omega_R)}+CR \norm{\vec g}_{L^\infty(\Sigma_R)}+CR^{-d/2} \norm{\vec u}_{L^2(\Omega_R)}.
\]
The proof is complete.
\end{proof}

\begin{lemma}				\label{lem:a2}
Let $\Omega \subset \bR^d$ be a bounded Lipschitz domain.
Let $\vec u\in W^{1,2}(\Omega)^m$ be a weak solution of the problem
\[
\left\{
\begin{aligned}
L\vec u&=\vec f \quad\text{in }\;\Omega,\\
\vec A D\vec u \cdot \vec n&= \vec g \quad \text{on }\;\partial\Omega,
\end{aligned}
\right.
\]
where $\vec f\in L^\infty(\Omega)^m$ and $\vec g \in L^\infty(\partial\Omega)^m$.
Then we have
\begin{equation}					\label{eq6.1ap}
 \norm{D\vec u}_{L^2(\Omega_{R/2})} \le C R^{-1}\norm{\vec u}_{L^2(\Omega_R)}+CR^{d/2}(1+R)\norm{\vec g}_{L^\infty(\Sigma_R)}+CR^{d/2+1}\norm{\vec f}_{L^\infty(\Omega_R)},
\end{equation}
where $C=C(d,m,\lambda, M,\Omega)$.
\end{lemma}
\begin{proof}
Let  $\eta$ be a smooth function on $\bR^d$ satisfying
\[
0\le \eta\le 1,\quad \supp \eta \subset B_R,\quad \eta\equiv 1\,\text{ on }\, B_{R/2},\quad\text{and}\quad \abs{D\eta} \le 4/R.
\]
By setting $\vec v=\eta^2 \vec u$ in \eqref{eq2.6ux}, we obtain
\[
 \int_\Omega \eta^2 A^{\alpha\beta}_{ij} D_\beta u^j D_\alpha u^i = -\int_\Omega 2\eta A^{\alpha\beta}_{ij} D_\beta \eta u^j D_\alpha u^i + \int_{\partial\Omega} \eta^2 g^i u^i + \int_\Omega \eta^2 f^i u^i.
\]
Then by \eqref{eqP-02}, \eqref{eqP-03}, and Cauchy's inequality, we get
\begin{equation}					\label{eq6.2ap}
\int_\Omega \eta^2 \abs{ D \vec u}^2 \le  C \int_\Omega\abs{D \eta}^2 \abs{\vec u}^2+C \norm{\vec g}_{L^\infty(\Sigma_R)} \norm{\eta\vec u}_{L^1(\partial\Omega)}+C\norm{\vec f}_{L^\infty(\Omega_R)}\norm{\vec u}_{L^1(\Omega_R)}.
\end{equation}
Observe that the trace theorem (see e.g., \cite{EG}) yields
\[
\int_{\partial\Omega} \abs{\eta \vec u} \le C\int_\Omega \abs{D(\eta \vec u)}+\abs{\eta \vec u} \le C\int_\Omega\abs{D\eta}\abs{\vec u}+\eta \abs{D\vec u}+\abs{\eta \vec u}.
\]
Therefore, by using H\"older's inequality and Cauchy's inequality, we estimate
\begin{align*}
 \norm{\vec g}_{L^\infty(\Sigma_R)}\norm{\eta\vec u}_{L^1(\partial\Omega)} &\le C\norm{\vec g}_{L^\infty(\Sigma_R)}\left( R^{d/2-1}\norm{\vec u}_{L^2(\Omega_R)}+ CR^{d/2}\norm{\eta D\vec u}_{L^2(\Omega_R)}+CR^{d/2}\norm{\vec u}_{L^2(\Omega_R)}\right)\\
 &\le CR^d(1+\epsilon^{-1}+R^2)\norm{\vec g}_{L^\infty(\Sigma_R)}^2+CR^{-2}\norm{\vec u}_{L^2(\Omega_R)}^2+\epsilon\int_\Omega \eta^2 \abs{ D \vec u}^2.
 \end{align*}
Similarly, by H\"older's inequality and Cauchy's inequality, we obtain
\[
\norm{\vec f}_{L^\infty(\Omega_R)}\norm{\vec u}_{L^1(\Omega_R)}\le CR^{d/2}\norm{\vec f}_{L^\infty(\Omega_R)}\norm{\vec u}_{L^2(\Omega_R)} \le CR^{d+2}\norm{\vec f}_{L^\infty(\Omega_R)}^2+ CR^{-2}\norm{\vec u}_{L^2(\Omega_R)}^2.
\]
 By combining \eqref{eq6.2ap} and the above inequality, we get
 \[
 \norm{D\vec u}_{L^2(\Omega_{R/2})}^2 \le C R^{-2}\norm{\vec u}_{L^2(\Omega_R)}^2+CR^d(1+R^2)\norm{\vec g}_{L^\infty(\Sigma_R)}^2+CR^{d+2}\norm{\vec f}_{L^\infty(\Omega_R)}^2.
 \]
 The above inequality obviously yields \eqref{eq6.1ap}.
\end{proof}

\begin{lemma}			\label{lem6.8ap}
Let $\Omega\subset\bR^d$ be a Lipschitz graph domain with Lipschitz constant $K$.
Then, for any $u \in Y^{1,2}(\Omega)$, we have
\begin{equation}			\label{eq05ap}
\norm{u}_{L^{2d/(d-2)}(\Omega)} \le C(d,K) \norm{D u}_{L^2(\Omega)}.
\end{equation}
\end{lemma}

\begin{proof}
We begin with showing that $C_c^\infty(\overline \Omega)$ is dense in $Y^{1,2}(\Omega)$.
By following the same steps as in the proof of approximation theorem for Sobolev functions (see, e.g., \cite{EG}), we find that $C^{\infty}(\overline \Omega)\cap Y^{1,2}(\Omega)$ is dense in $Y^{1,2}(\Omega)$.
On the other hand, one can approximate $u\in  C^{\infty}(\overline\Omega)\cap Y^{1,2}(\Omega)$ by a sequence of functions from $C_c^\infty(\overline\Omega)$ in the $Y^{1,2}(\Omega)$ norm as follows.
For each $k=1,2,\ldots$, let $\phi_k\in C_c^\infty(\bR^d)$ be such that
\[
0\le \phi_k \le 1,\quad \phi_k=1\;\text{ on }\;B_k(0),\quad \supp \phi_k\subset B_{3k}(0),\quad \text{and }\; \abs{\nabla\phi_k}\le 1/k.
\]
Then, obviously $u \phi_k\in C_c^\infty(\overline{\Omega})$ and it is easy to check $\norm{u \phi_k- u}_{Y^{1,2}(\Omega)} \to 0$ as $k \to \infty$.
We have thus shown that $C_c^\infty(\overline \Omega)$ is dense in $Y^{1,2}(\Omega)$.
By essentially the same argument, we also find that $C^\infty_c(\bR^d)$ is dense in $Y^{1,2}(\bR^d)$.
Therefore, the Sobolev inequality yields that
\begin{equation}					\label{eq5.1ys}
\norm{u}_{L^{2d/(d-2)}(\bR^d)} \le C(d) \norm{D u}_{L^2(\bR^d)},\quad \forall u\in Y^{1,2}(\bR^d).
\end{equation}
Next, we claim that there exists a bounded linear operator $E: Y^{1,2}(\Omega)\to Y^{1,2}(\bR^d)$
such that $Eu=u$ in $\Omega$ and 
\begin{equation}			\label{eq06ap}
\norm{D(Eu)}_{L^2(\bR^d)} \le C(d,K) \norm{Du}_{L^2(\Omega)}.
\end{equation}
To prove \eqref{eq06ap}, we follow the same steps in the usual proof of extension theorem for Sobolev functions in Lipschitz domain (see, e.g., \cite{EG}).
For $u\in C^{\infty}_c(\overline\Omega)$, set
\begin{align*}
u^+(y)&= u(y)\;\text{ if }\; y\in \overline \Omega,\\
u^-(y)&= u(y',2\gamma(y')-y_d)\;\text{ if }\; y\in \bR^d\setminus\Omega.
\end{align*}
Note $u^-=u^+=u$ on $\partial\Omega$.
Then, it is routine to check (see \cite[Section~4.1]{EG})
\[
\norm{D u^-}_{L^2(\bR^d\setminus\overline\Omega)} \le C(K) \norm{D u}_{L^2(\Omega)},
\]
and thus, we have
\[
\norm{u^-}_{Y^{1,2}(\bR^d\setminus\overline\Omega)} \le C(K) \norm{u}_{Y^{1,2}(\Omega)}.
\]
Define
\[
Eu\equiv\bar u\equiv
\left\{
\begin{aligned}
\;u^+&\quad\text{on}\;\overline\Omega,\\
\;u^-&\quad \text{on }\;\bR^d\setminus\Omega,
\end{aligned}
\right.
\]
and note that $\bar u$ is continuous on $\bR^d$.
Also, it is easy to see $\bar u \in Y^{1,2}(\bR^d)$ and 
\[
D\bar u=
\left\{
\begin{aligned}
\;Du^+&\quad\text{on}\;\overline\Omega,\\
\;Du^-&\quad \text{on }\;\bR^d\setminus\Omega.
\end{aligned}
\right.
\]
Therefore, we have proved \eqref{eq06ap} in the case when $u\in C_c^\infty(\overline\Omega)$.
Since $C_c^\infty(\overline\Omega)$ is dense in $Y^{1,2}(\Omega)$, we obtain \eqref{eq06ap} by the standard approximation argument.
Finally, we obtain \eqref{eq05ap} by combining \eqref{eq5.1ys} and \eqref{eq06ap}.
\end{proof}


\end{document}